\documentclass[a4paper,12pt]{article}
\usepackage{indentfirst}
\usepackage{graphicx}
\usepackage{epstopdf}
\usepackage{epsfig}
\usepackage{overpic}
\usepackage{overcite}
\usepackage{array}
\usepackage{color}
\usepackage{multirow}
\usepackage{float}
\usepackage{subfigure}
\usepackage{amsmath,amssymb, amsthm}
\usepackage[mathscr]{euscript}
\usepackage{latexsym,bm}
\usepackage{booktabs}
\usepackage{bbm}
\usepackage{diagbox}
 \usepackage[all]{xy}
 \usepackage{color}
\usepackage{diagbox}
\usepackage{algorithm}  
\usepackage{algorithmicx} 
\usepackage{algpseudocode} 
\usepackage{mathtools} 

\newtheorem{theorem}{Theorem}
\newtheorem{lemma}{Lemma}
\newtheorem{definition}{Definition}

\newcommand\comment[1]{}

\graphicspath{{images/}}

{\theoremstyle{remark} }
\newtheorem{proposition}{Proposition}

\newcommand\D{\textup{d}}

\def\mone{\mathbbm{1}}
\def\bx{\bm{x}}

\def\by{\bm{y}}
\def\bba{\bm{a}}
\def\bbb{\bm{b}}
\def\bu{\bm{u}}
\def\bv{\bm{v}}
\def\osc{\mathop{\textup{osc}}}
\makeatletter
\@addtoreset{equation}{section}
\makeatother

\begin{document}
\bibliographystyle{plain}

\title{SPADE: Sequential-clustering Particle Annihilation via Discrepancy Estimation}
\author{Sihong Shao\footnotemark[2], \and
Yunfeng Xiong\footnotemark[2]}
\renewcommand{\thefootnote}{\fnsymbol{footnote}}
\footnotetext[2]{LMAM and School of Mathematical Sciences, Peking University, Beijing 100871, China. Email addresses: {\tt sihong@math.pku.edu.cn} (S. Shao), {\tt xiongyf1990@pku.edu.cn} (Y. Xiong).}
\date{\today}
\maketitle

\begin{abstract}
For an empirical signed measure $\mu = \frac{1}{N} \left(\sum_{i=1}^P \delta_{\bx_i} - \sum_{i=1}^M \delta_{\by_i}\right)$, particle annihilation (PA) removes $N_A$ particles from both $\{\bx_i\}_{i=1}^P$ and $\{\by_i\}_{i=1}^M$ simultaneously, yielding another empirical signed measure $\nu$
such that $\int f \D \nu$ approximates to $\int f \D \mu$ within an acceptable accuracy for suitable test functions $f$. 
Such annihilation of particles carrying opposite importance weights has been extensively utilized for alleviating the numerical sign problem in particle simulations.
In this paper, we propose an algorithm for PA in high-dimensional Euclidean space based on hybrid of clustering and matching, dubbed the Sequential-clustering Particle Annihilation via Discrepancy Estimation (SPADE). It consists of two steps: Adaptive clustering of particles via controlling their number-theoretic discrepancies,  and independent random matching among positive and negative particles in each cluster. Both deterministic error bounds by the Koksma-Hlawka inequality and non-asymptotic random error bounds by concentration inequalities are proved to be affected by
two factors. One factor measures the irregularity of point distributions and reflects their discrete nature. The other relies on the variation of test function and is influenced by the continuity. Only the latter implicitly depends on dimensionality $d$, implying that SPADE can be 
immune to the curse of dimensionality for a wide class of test functions. Numerical experiments up to $d=1080$ validate our theoretical discoveries.
 \end{abstract}
 
 
 
 
 \vspace*{4mm}




\clearpage
 
\tableofcontents
 \newsavebox{\tablebox}

\section{Introduction}
\label{sec:intro}

Given two sequences $\mathcal{X} = (\bx_1, \dots, \bx_P) \subset \mathrm{Q}^P$ and $\mathcal{Y} = (\by_1, \dots, \by_{M})  \subset \mathrm{Q}^{M}$ in a finite rectangular domain $\mathrm{Q} \subset \mathbb{R}^d$,
$\mathcal{X}$ and $\mathcal{Y}$ are sets of positive and negative particles, carrying opposite weights $1$ and $-1$, respectively. 
An empirical signed measure $\mu$ is given by
\begin{equation}\label{def.empirical_signed_measure}
\mu = \frac{1}{N} \left( \sum_{i=1}^P \delta_{\bx_i} - \sum_{i=1}^M \delta_{\by_i} \right),
\end{equation}
where $N > 0$ is a prescribed normalizing constant and $\delta_{\bx}$ is Dirac measure concentrated at $\bx$. Typical choices of $N$ include $N = N_{tot}$ with $N_{tot} = P + M$ the total particle number or $N = P - M$ for $P > M$ in order to make $\mu(\mathrm{Q}) = 1$. 
Removing $N_A$ positive particles from $\mathcal{X}$ and $N_A$ negative ones from $\mathcal{Y}$ simultaneously yields two subsets $\mathcal{X}^A = (\tilde{\bx}_1, \dots, \tilde{\bx}_{P - N_A}) \subset \mathcal{X}$ and $\mathcal{Y}^A = (\tilde{\by}_1, \dots, \tilde{\by}_{M - N_A}) \subset \mathcal{Y}$ and another empirical signed measure
\begin{equation}
\nu = \frac{1}{N} \left( \sum_{i=1}^{P-N_A} \delta_{\tilde{\bx}_i} - \sum_{i=1}^{M - N_A} \delta_{\tilde{\by}_i} \right).
\end{equation}
Of direct interest in various applications is seeking a strategy to diminish the error function
\begin{equation}\label{def.error_function}
\mathcal{E}(f) =  \int f (\D \mu - \D \nu)
\end{equation}
for a suitable class of test functions $f$.

The aforementioned problem is called particle annihilation (PA) since positive and negative particles are paired and annihilated, so that the opposite importance weights are cancelled out. 
It is also known as particle cancelation or particle resampling on some occasions.
PA is originated from a vast number of Monte Carlo applications ranging from classical particle transport to quantum Monte Carlo.
In classical regime, the signed measure describes fluctuations over the equilibrium \cite{BakerHadjiconstantinou2005,YanCaflisch2015}. While in quantum mechanics, negative weights emerge in Monte Carlo integration of oscillatory and determinantal functions, including the path-integral representation of quantum observables \cite{BoothThomAlavi2009}, the pseudo-differential operator in the Wigner dynamics \cite{KosinaNedjalkovSelberherr2003,ShaoXiong2019} and the Rayleigh quotients with Slater determinant ansatz of fermionic wave-functions \cite{ReynoldsCeperleyAlderLester1982}.

Unfortunately, stochastic estimators of the form in Eq.~\eqref{def.empirical_signed_measure} usually suffer from the notorious numerical sign problem due to the near-cancelation of contributions from positive and negative weights, which stem from oscillating functions or even and odd permutations in determinants \cite{DuBoisBrownAlder2017}. As a consequence, it leads to an exponential increase of particle number and stochastic variance, as well as the computational complexity, especially when simulating long-time particle dynamics \cite{YanCaflisch2015,ReynoldsCeperleyAlderLester1982,BoothThomAlavi2009,ShaoXiong2019}.
The sign problem is believed to be NP-hard since there might not exist an algorithm to achieve a relative statistical error scaling polynomially with the dimensionality $d$ and the system size in general \cite{TroyerWiese2005,IazziSoluyanovTroyer2016}. A well-known example is given in \cite{TroyerWiese2005} by mapping a 3-D Ising spin glass model, which belongs to the complexity class NP, to a quantum system with the sign problem.

Many efforts have been made to potentially overcome the numerical sign problem in certain situations, albeit usually in the settings that rely heavily on concrete problems. Several related approaches include the particle resampling in kinetic theory by filtering out high-frequency components \cite{YanCaflisch2015}, the grid-based annihilation in quantum Boltzmann simulations \cite{KosinaNedjalkovSelberherr2003,ShaoXiong2019}, the fixed-node approximation in diffusion Monte Carlo \cite{ReynoldsCeperleyAlderLester1982}, the annihilation of determinants in Fermion Monte Carlo \cite{BoothThomAlavi2009}, the resummation in an inchworm algorithm \cite{CaiLuYang2020} and the stationary phase approximation in the Wigner branching random walk \cite{ShaoXiong2019_ArXiv}. In spite of their distinct appearances, the core of all the above approaches is essentially the same, that is,  to fully utilize the cancelation of positive and negative weights, thereby reducing the growth of particle number as well as stochastic variance. This creed
will be faithfully inherited in subsequent trials for a new approach. 


Our approach, termed the Sequential-clustering Particle Annihilation via Discrepancy Estimation (SPADE), hybridizes sequential clustering of particles and independent random matching in each cluster, as depicted by the diagram:
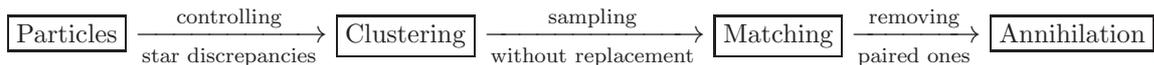
\begin{equation*}
\label{diagram}
\boxed{\footnotesize \text{Particles}} \xrightarrow[\textup{star discrepancies}]{\textup{controlling}} \boxed{\footnotesize \text{Clustering}}\xrightarrow[\textup{without replacement}]{\textup{sampling}}\boxed{\footnotesize \text{Matching}}  \xrightarrow[\textup{paired ones}]{\textup{removing}} 
\boxed{\footnotesize \text{Annihilation}}
\end{equation*}
The SPADE algorithm utilizes both number-theoretic and statistical properties of point distributions. Detailedly speaking, SPADE consists of two steps. The first step is to perform adaptive clustering of positive and negative particles through a sequential binary partition of the domain $\mathrm{Q}$. Each time we pick up a cluster of particles and make a further split until their star discrepancies, say, the number-theoretic irregularities, reach a prescribed threshold, partially borrowing the pioneering idea from the non-parametric density estimation method based on discrepancy sequential partitioning \cite{LiWangWong2016}. Then it succeeds at the second step by seeking independent random matchings in each cluster, which is essentially equivalent to sampling from a finite population without replacement, and removing paired particles. In contrast to existing approaches for PA, such as `fixed-node approximation' \cite{ReynoldsCeperleyAlderLester1982}, SPADE requires no a priori information of the exact nodal hypersurface of underlying integrands. Instead, it will `learn' the nodes based on the point distribution in an adaptive manner. In a sense, SPADE can be intuitively treated as a `flexible-node approximation'.

From the perspective of numerical analysis, it requires to bound the error function $\mathcal{E}(f)$ and to ensure its convergence as $N \to \infty$.
The interplay between the continuous integrals and discrete (combinatorial) nature of point distributions will be extensively utilized in analysis of SPADE, and the discrepancy theory directly links both sides. First, the deterministic error bounds under arbitrary matching are given with the help of the famous Koksma-Hlawka inequality \cite{KuipersNiederreiter2012,DrmotaTichy2006} and a combinatorial property of the star discrepancy. The obtained bound can be separated into two parts: one comes from irregularities of point distributions (measured by the star discrepancy) and the other from the variation of test function $f$ in the sense of Hardy and Krause, without any a priori knowledge of continuity. Since only the latter implicitly depends on the dimensionality $d$, the error bounds may be immune to the curse of dimensionality for some good test functions. This explains why SPADE is able to achieve a satisfactory numerical accuracy for certain problems up to $d=1080$ and work well for slow-varying test functions without continuous derivates.
Second, non-asymptotic stochastic error bounds under random matching 
further sharpens the deterministic ones with the help of concentration inequalities for the sum of independent random variables \cite{Bennett1962,Hoeffding1963,bk:BoucheronLugosiMassart2013} and sampling from a finite population without replacement \cite{Serfling1974,BardenetMaillard2015,Tolstikhin2017}.
As a comparison, the optimal assignment \cite{Kuhn1955} is also adopted to pair the positive and negative particles in each cluster. Numerical results show that the random matching usually outperforms the optimal assignment.


The rest is organized as follows. Notations and definitions are given in Section \ref{notation_def}, proceeding with a statement of main results in Section \ref{sec:result}.  The sequential clustering is issued in Section \ref{sec:clustering} with a proof of deterministic bounds.  The random matching is discussed in Section \ref{sec:matching} and the bounds of error function will be sharpened thanks to concentration of random variables. It follows by numerical experiments in Section \ref{sec:numerical}. 
The paper is concluded in Section \ref{sec:conclusion}
with a few remarks.

\bigskip

{\bf Acknowledgements.}
This research was supported by the National Natural Science Foundation of China (Nos.~11822102, 11421101) and High-performance Computing Platform of Peking University. 
SS is partially supported by Beijing Academy of Artificial Intelligence (BAAI). 
YX is partially supported by The Elite Program of Computational and Applied Mathematics for PhD Candidates in Peking University.

\section{Notations and definitions}
\label{notation_def}


We use the short hand notation $\mathrm{Q} = [\bba, \bbb] = \prod_{j=1}^d [a_j, b_j]$ to denote a hyper-rectangle in $\mathbb{R}^d$, where $\bba = (a_1, \dots, a_d)$ and $\bbb = (b_1, \dots, b_d)$ with $a_j$ and $b_j$ specifying the lower and upper bound of $\mathrm{Q}$ along dimension $j$.  A {\bf partition} of $\mathrm{Q}$ at level $K$ is denoted by $\mathcal{P}_K = \{\mathrm{Q}_1, \dots, \mathrm{Q}_K\}$, where $\mathrm{Q} = \bigcup_{k=1}^K \mathrm{Q}_k$ and $\mathrm{Q}_k = [\bba_k, \bbb_k]$ are mutually disjoint rectangular bins, say, $\mathrm{Q}_i \bigcup \mathrm{Q}_j = \varnothing$ for $i \ne j$. Equivalently, it means a set of $d$ finite sequences $\eta_i^{(0)}, \dots, \eta_i^{(m_i)}, i = 1, \dots, d$, with
\[
a_i \le \eta_i^{(0)} \le \cdots \le \eta_i^{(m_i)} \le b_i, 
\]
by sorting the $i$-th coordinates of all the lower and upper bounds of $\mathrm{Q}_k$ in an increasing order. The collection of all partitions is denoted by $\mathcal{P}$.  

The {\bf subsequences} of $\mathcal{X}$ and $\mathcal{Y}$ in $\mathrm{Q}_k$ are denoted by $\mathcal{X}_k = (\bx_{1}^{(k)}, \dots, \bx_{P_k}^{(k)})$ and $\mathcal{Y}_k =  (\by_{1}^{(k)}, \dots, \by_{M_k}^{(k)})$, respectively. 
$P_k = |\mathcal{X}_k|$ and  $M_k = |\mathcal{Y}_k|$ are the cardinal numbers,
and we let $P_k \wedge M_k = \min\{ P_k, M_k\}$. 


\begin{definition}
For two finite sequences $(\bx_1, \dots, \bx_P)$ and $(\by_1, \dots, \by_M)$, $P \ge M$, a {\bf matching} is an injective function $\sigma: \{1, \dots, M\} \to \{1, \dots, P\}$. A random matching $\sigma$ satisfies
\begin{equation}\label{probability_law_sampling_without_replacement}
\Pr\left\{(f(X_1), \dots, f(X_t)) = (f(\bx_{\sigma(1)}), \dots, f(\bx_{\sigma(t)}) \right\} = \frac{1}{P(P-1)\cdots(P-t+1)}
\end{equation}
for random variables $X_t$, $1\le t \le M$, which is equivalent to sampling $M$ times from a finite population $\{1, \dots, P\}$ without replacement. By contrast, an optimal assignment is a matching $\sigma$ to attain $\min_{\sigma} \sum_{i=1}^M  \Vert \bx_{\sigma(i)} - \by_i \Vert$ for a given norm.
\end{definition}

To quantitively measure the difference between the continuous integral and the arithmetic mean of the sequences $\{ f(\bx_1), \dots, f(\bx_P) \}$, i.e., 
\begin{equation}
 \Big | \frac{1}{\lambda(\mathrm{Q})}\int_{\mathrm{Q}} f(\bx) \D \bx - \frac{1}{P} \sum_{i=1}^P f(\bx_i) \Big |,
\end{equation}
we need both concepts of the variation of the function $f$ and the star discrepancy of sequence $(\bx_1, \dots, \bx_P)$. Here $\lambda(\mathrm{Q})$ gives the Lesbesgue measure of $\mathrm{Q}$.

Given a partition of $\mathrm{Q}$, we define an operator 
\begin{equation}
\begin{split}
\Delta_i f(x_1, \dots, x_{i-1}, \eta_{i}^{(j)}, x_{i+1}, \dots, x_d) = & f(x_1, \dots, x_{i-1}, \eta_{i}^{(j+1)}, x_{i+1}, \dots, x_d) \\
&- f(x_1, \dots, x_{i-1}, \eta_{i}^{(j)}, x_{i+1}, \dots, x_d).
\end{split}
\end{equation}
The {\bf variation} of $f$ on $[\bba, \bbb]$ in the sense of Hardy and Krause, denoted by $V_{[\bba, \bbb]}^{HK}(f)$ \cite{DrmotaTichy2006, Owen2005}, is delineated as follows. 

\begin{definition}\label{def.variation}
The variation of function $f$ on $[\bba, \bbb]$ in the sense of Vitali reads 
\begin{equation}
V^{(l)}(f) = \sup_{\mathcal{P}} \sum_{j_1 = 0}^{m_1 - 1} \cdots \sum_{j_d = 0}^{m_d - 1} \Big | \Delta_1 \cdots \Delta_l f(\eta_1^{(j_1)}, \dots, \eta_{d}^{(j_d)})\Big |,
\end{equation}
where the supremum is extended over all possible partitions $\mathcal{P}$. 
Then the variation of $f$ on $[\bba, \bbb]$ in the sense of Hardy and Krause
is 
\begin{equation}\label{def.variation_continuous}
V_{[\bba, \bbb]}^{HK}(f)= \sum_{l=1}^{d} \sum_{F_l} V^{(l)}(f^{(F_l)}),
\end{equation}
where $f^{(F_l)}$ presents the restriction of $f$ to the $l$-dimensional face $F_l$ of dimensions $\{i_1, i_2, \dots, i_{l}\} \subset \{1, \dots, d\}$, defined by $x_i = b_i$ for $i \notin \{i_1, \dots, i_l\}$, and the second sum is extended over all $l$-dimensional faces $F_l$. In particular, for continuous functions $f$ with continuous mixed partial derivates, we have 
\begin{equation}\label{variation_continuous_mixed_derviate}
V_{[\bba, \bbb]}^{HK}(f) = \sum_{l=1}^{d} \sum_{F_l}  \int_{F_l} \Big | \frac{\partial^l f^{(F_l)}}{\partial x_{i_1} \dots \partial x_{i_l}} \Big | \D x_{i_1} \dots \D x_{i_l}.
\end{equation}
\end{definition}

For functions with vanishing mixed derivates, e.g., $f(\bx) = x_1 + \dots + x_d$ or $f(\bx) = x_1^2 + \dots + x_d^2$, $V_{[\bba, \bbb]}^{HK}(f)$ may depend on $d$ linearly (see Eqs.~\eqref{f1} and \eqref{f2}). By contrast, for $f(\bx) = x_1 x_2 \cdots x_d$, $V_{[\bba, \bbb]}^{HK}(f)$ depends on $d$ exponentially (see Eq.~\eqref{f5}). 
For the indicating function $f(\bx) = \chi_I(\bx)$,
it has that $V_{[\bba, \bbb]}^{HK}(f) = 2^d$ for 
$I = [\tilde{\bba}, \tilde{\bbb}] \subset [\bba, \bbb]$ with $a_i < \tilde{a}_i < \tilde{b}_i < b_i$, $1\le i \le d$  \cite{Owen2005}.

Another characterization of local variation of $f$ on $[\bba, \bbb]$ is the {\bf oscillator norm}:\begin{equation}\label{def.bounded_difference}
\osc_{[\bba, \bbb]}(f) = \sup_{(\bx, \bx^{\prime}) \in [\bba, \bbb]^2} |f(\bx) - f(\bx^{\prime})|,
\end{equation}
and the {\bf local $L^2$-norm} of $f$ on $\mathrm{Q}$ is defined by
\begin{equation}
\Vert f \Vert_{L^2(\mathrm{Q})} = \left(\frac{1}{\lambda(\mathrm{Q})} \int_{\mathrm{Q}} f^2(\bx) \D \bx\right)^{1/2}.
\end{equation}

The {\bf star discrepancy} \cite{KuipersNiederreiter2012, DrmotaTichy2006}, stemming from number theory, is adopted to measure the irregularity of a sequence $(\bx_1, \dots, \bx_P) \subset [0,1]^d$.
\begin{definition}
Let $\mathcal{X} = (\bx_1, \dots, \bx_P)$ be a finite sequence of points in the $d$-dimensional space $[0,1]^d$. Then the number
\begin{equation}\label{def.star_discrepancy}
D^\ast_P(\bx_1, \dots, \bx_P) = \sup_{\bu \in [0, 1]^d} \Big |\frac{A([\bm{0}, \bu), P, \mathcal{X})}{P}  - \lambda([\bm{0}, \bu))\Big |,
\end{equation}
is called the star discrepancy of sequence $\mathcal{X}$.
Here $A(I, P, \mathcal{X}) = \sum_{i=1}^P \chi_{I}(\bx_i)$
gives the number of the sequence $\mathcal{X}$ in $I$.
\end{definition}
For a general sequence $(\bx_1, \dots, \bx_P) \subset [\bba_k, \bbb_k]$, its irregularity can be measured by the star discrepancy after mapping it from  $[\bba_k, \bbb_k]$ to $[0, 1]^d$ via a linear {\bf scaling}:
\begin{equation}\label{def.scaling_mapping}
\phi_{k}(\bx) = \left(\frac{x_1 - a_{k, 1}}{b_{k, 1} - a_{k, 1}}, \frac{x_2 - a_{k, 2}}{b_{k, 2} - a_{k, 2}}, \dots, \frac{x_d - a_{k, d}}{b_{k, d} - a_{k, d}}\right).
\end{equation}

In this work, we always deal with test functions in the class with bounded variations in the sense of Hardy and Krause, say,  $V_{[\bba, \bbb]}^{HK}(f) < \infty$. Under the partition $\mathcal{P}_K$, their oscillator norms $\osc_{[\bba_k, \bbb_k]}(f)$ are also bounded for $k=1,\dots,K$.

\section{Statement of main theorems}
\label{sec:result}

The SPADE algorithm consists of two steps: Clustering and matching.

{\bf Clustering}: Establish a partition $\mathcal{P}_K$ of $\mathrm{Q}$ of level $K$ associated with clustering of particles $\mathcal{X}$ and $\mathcal{Y}$ such that 
the star discrepancies of both clustered subsequences $\mathcal{X}_k$ and $\mathcal{Y}_k$ in $\mathrm{Q}_k$ are all bounded for $k=1,\dots,K$:  
\begin{equation}\label{def.discrepancy_bound}
\begin{split}
D_{P_k}^\ast(\phi_{k}(\bx_{1}^{(k)}), \dots, \phi_{k}(\bx_{P_k}^{(k)})) &\le \frac{\vartheta \sqrt{N}}{P_k}, \\
D_{M_k}^\ast(\phi_{k}(\by_{1}^{(k)}), \dots, \phi_{k}(\by_{M_k}^{(k)})) &\le \frac{\vartheta \sqrt{N}}{M_k},
\end{split}
\end{equation}
where $\vartheta \in (0, 1)$ is a prescribed parameter to control the error function $\mathcal{E}(f)$.


{\bf Matching:} Within each cluster, seek a matching $\sigma$  from $\{1, \dots, M_k\}$ to $\{1, \dots, P_k\}$  independently when $M_k \le P_k $, or from $\{1, \dots, P_k\}$ to $\{1, \dots, M_k\}$ when $M_k > P_k$, and then remove the matched pairs from $\mathcal{X}_k$ and $\mathcal{Y}_k$.

Under the partition $\mathcal{P}_K$, 
the error function between two empirical signed measures $\mu$ (before annihilation) and $\nu$ (after annihilation) is 
\begin{align}
\mathcal{E}(f) &= \sum_{k=1}^K \mathcal{E}_k(f), \\
\mathcal{E}_k(f) &= \xi_k + \frac{P_k \wedge M_k}{N P_k } \sum_{i=1}^{P_k} f(\bx_i) - \frac{P_k \wedge M_k}{N M_k} \sum_{i=1}^{M_k} f(\by_i), \label{def.local_error_SPADE} \\
\xi_k &= 
\begin{cases}
\displaystyle 
\frac{P_k \wedge M_k}{N} \left(\frac{1}{M_k} \sum_{i=1}^{M_k} f( \bx_{\sigma(i)}) - \frac{1}{P_k} \sum_{i=1}^{P_k} f(\bx_i)\right),  & P_k \ge M_k, \\
\displaystyle
  \frac{P_k \wedge M_k}{N}  \left(\frac{1}{M_k} \sum_{i=1}^{M_k} f( \by_i) - \frac{1}{P_k} \sum_{i=1}^{P_k} f(\by_{\sigma(i)})\right),  & P_k < M_k. 
\end{cases}
 \label{def.xi}
\end{align}

The condition \eqref{def.discrepancy_bound} is used to bound the second and third terms in Eq.~\eqref{def.local_error_SPADE} owing to the Koksma-Hlawka inequality. As for the first term, the bound can be obtained by either the combinatorial property of the star discrepancy or the bounded difference condition $\osc_{[\bba_k, \bbb_k]}(f) < \infty$. It deserves to mention that $K$, the partition level, depends on $d$ implicitly as the star discrepancy does. In general, $K$ has to be smaller than total particle number $N_{tot}$ to avoid the `overfitting' phenomenon.

\begin{theorem}[Deterministic bounds under arbitrary matching]
\label{thm.particle_annhilation_error_discrete_bound}
For two sequences $\mathcal{X}$ and $\mathcal{Y}$ under partition $\mathcal{P}_K$, 
suppose (i) there exists a constant $0 < \vartheta \le 1$ such that the bounds of star discrepancies  \eqref{def.discrepancy_bound} hold for each cluster; (ii) there exists another constant $\gamma \in [\vartheta,1]$ such that $P_k \wedge M_k \le \gamma \sqrt{N}$ for each cluster. Then under arbitrary matching between $\mathcal{X}_k$ and $\mathcal{Y}_k$, for any function $f$ with bounded variation $V_{[\bba, \bbb]}^{HK}(f) < \infty$ in the sense of Hardy and Krause, it has that 
\begin{equation}\label{def.HK_bounded_error_bound}
|\mathcal{E}(f)|  \le \frac{H_0(\vartheta, \gamma) }{N^{1/2}}V_{[\bba, \bbb]}^{HK}(f), 
\end{equation}
and
\begin{equation}\label{def.HK_bounded_error_bound_2}
|\mathcal{E}(f)|  \le \frac{\gamma}{4 N^{1/2}} \sum_{k=1}^K \osc_{[\bba_k, \bbb_k]}(f) + \frac{2 \vartheta }{N^{1/2}}V_{[\bba, \bbb]}^{HK}(f).
\end{equation}
Here $H_0(\vartheta, \gamma) =  \frac{\gamma}{4} + \frac{3\vartheta}{2} + \frac{\vartheta^2}{4\gamma}$.
\end{theorem}

The first bound \eqref{def.HK_bounded_error_bound}, obtained using the Koksma-Hlawka inequality and the combinatorial property of star discrepancies, characterizes the basic convergence of SPADE. A key observation follows that the sole factor that may implicitly depend on dimensionality $d$ is the variation function $V_{[\bba, \bbb]}^{HK}(f)$. In other words, the bound \eqref{def.HK_bounded_error_bound} might not be bothered by the curse of dimensionality for slow-varying test functions. 

The worse-case bound of $\sum_{k=1}^K \xi_k$ appears in the second bound \eqref{def.HK_bounded_error_bound_2}, which turns out to be a gross over-estimate in general.
A useful improvement is to adopt independent random matchings in each cluster, so that $\xi_1, \dots, \xi_K$ become independent random variables with $\mathbb{E}\xi_k = 0$ and the sum $\sum_{k=1}^K \xi_k$ will be concentrated near the mean value zero. Meanwhile, the second and third terms in Eq.~\eqref{def.local_error_SPADE} remain non-stochastic and are bounded by the Koksma-Hlawka inequality. In consequence,  stochastic upper bounds can be obtained owing to the concentration bounds of Hoeffding's type and Bernstein's type for sampling with or without replacement \cite{bk:BoucheronLugosiMassart2013,BardenetMaillard2015}.  
The stochastic bounds  can dramatically sharpen the deterministic worse-case bound with a high probability, albeit not tight.
Furthermore, estimation of the stochastic variance $\textup{Var}(\mathcal{E}(f))  = \mathbb{E}(\mathcal{E}(f) -  \mathbb{E}\mathcal{E}(f) )^2$ is also of the great importance for the random matching. That is equivalent to estimate 
\begin{equation}\label{def.variance_random_matching}
\textup{Var}(\mathcal{E}(f))  = \sum_{k=1}^K  \mathbb{E}(\mathcal{E}_k(f) -  \mathbb{E}\mathcal{E}_k(f) )^2 =  \sum_{k=1}^K  \mathbb{E}(\xi_k -  \mathbb{E}\xi_k )^2  = \sum_{k=1}^K \mathbb{E} \xi_k^2.
\end{equation}


\begin{theorem}[Stochastic bounds under random matching]
\label{thm.particle_annhilation_error_random_bound_no_variance}
Under the assumptions in Theorem \ref{thm.particle_annhilation_error_discrete_bound} and 
independent random matching between $\mathcal{X}_k$ and $\mathcal{Y}_k$ in each cluster, for any function $f$ with bounded variation $V_{[\bba, \bbb]}^{HK}(f) < \infty$ in the sense of Hardy and Krause, it holds that
\begin{equation}\label{random_expection_bound}
| \mathbb{E} \mathcal{E}(f) | \le \frac{2 \vartheta }{N^{1/2}}V_{[\bba, \bbb]}^{HK}(f),
\end{equation}
and 
\begin{equation}\label{random_variance_bound}
\textup{Var}(\mathcal{E}(f)) \le \frac{H_1(N, \gamma)}{N^{3/2}} \sum_{k=1}^K \left( \osc_{[\bba_k, \bbb_k]}(f)\right)^2,
\end{equation}
where $H_1(N, \gamma) = \frac{\gamma}{4} + \frac{1}{2N^{1/2}}+ \frac{1}{4\gamma N}$.
Furthermore, for any $\varepsilon \in (0, 1)$, with probability higher than $1 - \varepsilon$, it has that
\begin{equation}\label{random_error_Bernstein_bound}
\begin{split}
|\mathcal{E}(f)| \le \sqrt{2 \log(2/\varepsilon)\textup{Var}(\mathcal{E}(f)) } + \frac{\gamma \log(2/\varepsilon)}{6N^{1/2}} \max_{k} \osc_{[\bba_k, \bbb_k]}(f) + \frac{2\vartheta}{N^{1/2}} V_{[\bba, \bbb]}^{HK}(f).
\end{split}
\end{equation}
\end{theorem}

Eqs.~\eqref{random_expection_bound} and \eqref{random_variance_bound} characterize the bounds for the bias and variance of random matching, respectively. 
In the sense of expectation, the bias of PA is totally determined by the irregularity of point distributions as $\mathbb{E}\xi_k$ vanishes under sampling without replacement. It notes that the constant factor is tighter compared to Eq.~\eqref{def.HK_bounded_error_bound} as ${\gamma}/{4} + {\vartheta^2}/{4\gamma} \ge \vartheta/2$.
The variance of 
random matching is of the order $N^{-3/2}$ and the constant factor relies on the oscillator norm of $f$. 

The bound \eqref{random_error_Bernstein_bound} may dramatically improve \eqref{def.HK_bounded_error_bound_2}. For the first term, the constant factor is sharpened as 
$\sum_{k=1}^K \left(\osc_{[\bba_k, \bbb_k]}(f)\right)^2 \le \left(\sum_{k=1}^K \osc_{[\bba_k, \bbb_k]}(f) \right)^2$ and $\osc_{[\bba_k, \bbb_k]}(f) \ge 0$.
Actually, it gains an extra rate of $N^{-1/4}$. For the second term, the constant factor is diminished 
provided  
\[
 \log(2/\varepsilon)\max_{k} \osc_{[\bba_k, \bbb_k]}(f) \le \frac{3}{2} \sum_{k=1}^K   \osc_{[\bba_k, \bbb_k]}(f).
\]
The third term keeps unchanged due to the bias of matching. Therefore, with a high probability, the random matching is advantageous due to the concentration property near the mean value $\sum_{k=1}^K \mathbb{E} \xi_k = 0$.

We shall provide an alternative estimation of $\textup{Var}(\mathcal{E}(f))$ by taking the irregularity of point distribution into account under an additional assumption that $f^2$ has a bounded variation.   
 The discrepancy theory will be again used for bounding the local variance of random matching in each cluster, which in turn give an estimate of the second moment $\sum_{k=1}^K  \mathbb{E} \xi_k^2$.

\begin{theorem}[Variance estimation via bounding discrepancies]
\label{thm.particle_annhilation_error_random_bound_variance}
Under the assumptions in Theorem \ref{thm.particle_annhilation_error_discrete_bound} and Theorem \ref{thm.particle_annhilation_error_random_bound_no_variance}, for any function $f$ with bounded variations $V_{[\bba, \bbb]}^{HK}(f) < \infty$ and $V_{[\bba, \bbb]}^{HK}(f^2) < \infty$ in the sense of Hardy and Krause, it holds that
\begin{equation}\label{random_variance_bound_disc}
\begin{split}
\textup{Var}(\mathcal{E}(f)) \le  &\frac{\vartheta H_2(N)}{N^{1/2}(N+1)} \left( V_{[\bba, \bbb]}^{HK}(f^2) +  V_{[\bba, \bbb]}^{HK}(f) \cdot \sup_{[\bba, \bbb]}|f|\right) \\
  + & \frac{\gamma H_2(N)}{N^{1/2}(N+1)} \sum_{k=1}^K  \Vert f \Vert_{L^2([\bba_k, \bbb_k])} \cdot \osc_{[\bba_k, \bbb_k]}(f) + \frac{H_3(\gamma, N)}{N(N+1)} \sum_{k=1}^K \left(\osc_{[\bba_k, \bbb_k]}(f)\right)^2,
\end{split}
\end{equation}
where $H_2(N) = \frac{9\log(2(N+1))}{2}$ and $H_3(\gamma, N) = \frac{2(8+3\sqrt{2})^2 \log^2(2(N+1))}{9} + \frac{\gamma^2}{16}$.
\end{theorem}

By a comparison between the bounds \eqref{random_variance_bound} and \eqref{random_variance_bound_disc}, it is seen that the constant factor in the leading term of order $N^{-3/2}$ is replaced by a term composed of $\vartheta V_{[\bba, \bbb]}^{HK}(f^2)$, $\vartheta V_{[\bba, \bbb]}^{HK}(f) \cdot \sup_{[\bba, \bbb]}|f|$ and $\sum_{k=1}^K  \Vert f \Vert_{L^2([\bba_k, \bbb_k])} \cdot \osc_{[\bba_k, \bbb_k]}(f)$, which may be tighter than $\sum_{k=1}^K \left( \osc_{[\bba_k, \bbb_k]}(f)\right)^2$ when $\vartheta$ is sufficiently small and $|f| < 1$. The last term in Eq.~\eqref{random_variance_bound_disc} may be neglected due to the extra order of $N^{-1/2}$. Inserting Eq.~\eqref{random_variance_bound_disc} into Eq.~\eqref{random_error_Bernstein_bound} yields another concentration bound for random matching.

\section{Clustering and discrepancy theory}
\label{sec:clustering}

As the first and crucial step of the SPADE algorithm,
the clustering of particles is to divide the PA problem into several local parts. The prototype is the gird-based annihilation \cite{BakerHadjiconstantinou2005,KosinaNedjalkovSelberherr2003,ShaoXiong2019}, in which the domain is partitioned into bins of equal size, so that the positive and negative particles located in the same bin are annihilated. Apparently, the advantage of such method is its simplicity and the numerical errors are controlled by the bin size \cite{XiongShao2019}. It is also closely related to the histogram statistics in non-parametric density estimation since the annihilation can be realized based on the difference of two probability densities postulated from positive and negative particles, respectively.

However, the setting of uniform grid is no longer applicable for higher dimensional PA problem arising in particle simulations of many-body problems. The reasons are as twofold. First, the number of bins grows exponentially in $d$, known as the curse of dimensionality. Second, the required sample size for given accuracy also grows with the increasing number of bins \cite{XiongShao2019}. A potential way to resolve these challenges is to replace the uniform grid mesh by an adaptive one driven by the distribution of particle populations, borrowing the basic idea from high-dimensional statistical learning theory \cite{LuJiangWong2013,LiWangWong2016}. However, the negative weights emerged in the PA problem make it essentially different from the density estimation. 

In practice, an adaptive partition is established through the sequential binary splitting, which in turn clusters positive and negative particles in a hierarchical manner. Two key ingredients must be specified for the adaptive partition. One is the criterion for stopping the binary splitting. SPADE tries to utilize the number-theoretical properties (discrepancy) of two kinds of particles to make the decision like in discrepancy sequential partitioning \cite{LiWangWong2016}, instead of relying on any a priori probability distribution as adopted in the Bayesian interference \cite{LuJiangWong2013}. The other is the choice of nodes to be split, which is relevant in monitoring the nodal hyper-surfaces of underlying integrand function.  SPADE chooses a node to maximize a gap function between the distributions of positive and negative particles.

\subsection{Setting of sequential clustering}

\begin{algorithm}[t]
\caption{Clustering via adaptive partitioning}\label{algorithm_DED}
\begin{algorithmic}
\Function  {($\{\mathcal{X}_1,\dots, \mathcal{X}_K\}$, $\{\mathcal{Y}_1,\dots, \mathcal{Y}_K\}$) =Clustering}{$\mathcal{X}, \mathcal{Y}, \mathrm{Q}, \vartheta$} 
 \State $K=1$, $\mathrm{Q}_1 = \mathrm{Q}$, $\mathcal{X}_1 = \mathcal{X}$, $\mathcal{Y}_1 = \mathcal{Y}$, $\mathcal{P} = \{\mathrm{Q}_1\}$, $\mathcal{P}^{\prime}  =  \varnothing$.   \While{ $\mathcal{P} \ne  \mathcal{P}^{\prime}$ } 
   \State $\mathcal{P}^{\prime} = \mathcal{P}$
   \ForAll{$\mathrm{Q}_k = [\bba_k, \bbb_k]$ in $\mathcal{P}$, $k \le K$} 
   \State $\mathcal{P} \leftarrow \mathcal{P} \setminus \mathrm{Q}_k$, $\mathcal{X} \leftarrow \mathcal{X} \setminus \mathcal{X}_k$, $\mathcal{Y} \leftarrow \mathcal{Y} \setminus \mathcal{Y}_k$  
   \State Scale all $\bx_{i} \in \mathrm{Q}_k$ to  $ \phi_{k}(\bx_i) \in [0, 1]^d $ and $\by_{i} \in \mathrm{Q}_k$ to  $ \phi_{k}(\by_i) \in [0, 1]^d$ 
   \State Calculate the star discrepancy $\textup{disc}_+ = D_{P_k}^{\ast}(\phi_{k}(\bx^{(k)}_1), \dots, \phi_{k}(\bx^{(k)}_{P_k}))$  
   \State Calculate the star discrepancy $\textup{disc}_- = D_{M_k}^{\ast}(\phi_{k}(\by^{(k)}_1), \dots, \phi_{k}(\by^{(k)}_{M_k}))$
   \If {$\textup{disc}_+ >  \frac{\vartheta \sqrt{N}}{P_k} ~ \textup{or} ~ \textup{disc}_-  >  \frac{\vartheta \sqrt{N}}{M_k}$}  
   \State $K \leftarrow K + 1$
   \State Choose a node $c_{k, j}$ to attain the maximum of Eq.~\eqref{def.gap} 
   \State Divide the sub-region $\mathrm{Q}_k$ into $\mathrm{Q}_k^{(1)}$ and $\mathrm{Q}_k^{(2)}$ as given in Eq.~\eqref{split_subrectangle}
   \State Divide the pointsets: $\mathcal{X}_k = \mathcal{X}_k^{(1)} \cup \mathcal{X}_k^{(2)}$ and $\mathcal{Y}_k = \mathcal{Y}_k^{(1)} \cup \mathcal{Y}_k^{(2)}$
   \State Update the partition:  $\mathcal{P} \leftarrow \mathcal{P} \cup \mathrm{Q}_k^{(1)} \cup \mathrm{Q}_k^{(2)}$,
   \State Update the particle sets:  $\mathcal{X} \leftarrow \mathcal{X} \cup \mathcal{X}_k^{(1)} \cup \mathcal{X}_k^{(2)}$, $\mathcal{Y} \leftarrow \mathcal{Y} \cup \mathcal{Y}_k^{(1)} \cup \mathcal{Y}_k^{(2)}$
   \Else
   \State $\mathcal{P} \leftarrow \mathcal{P} \cup \mathrm{Q}_k$, $\mathcal{X} \leftarrow \mathcal{X} \cup \mathcal{X}_k$, $\mathcal{Y} \leftarrow \mathcal{Y} \cup \mathcal{Y}_k$
   \EndIf
   \EndFor
   \EndWhile
   \State \Return{$\mathcal{X} = \{\mathcal{X}_1,\dots, \mathcal{X}_K\}$, $\mathcal{Y} = \{\mathcal{Y}_1, \dots, \mathcal{Y}_K\}$, $\mathcal{P}_K = \mathcal{P}  = \{\mathrm{Q}_1, \dots, \mathrm{Q}_K\}$}  
\EndFunction  
\end{algorithmic}
\end{algorithm}

An adaptive partition is built up through the binary splitting. The binary partition can be defined in the following recursively way. It starts with $\mathcal{P}_0 = \mathrm{Q}$. Suppose at level $K$ we have $\mathcal{P}_K = \{\mathrm{Q}_1, \mathrm{Q}_2, \dots, \mathrm{Q}_K\}$ with $\mathrm{Q}_k$ mutually disjoint and $\mathrm{Q} = \cup_{k=1}^{K} \mathrm{Q}_k$, then the binary partition at level $K+1$ is obtained by choosing one sub-rectangle $\mathrm{Q}_k$ and dividing it into two parts along one of its coordinate, parallel to one of dimensions. 

For example, for $\mathrm{Q}_k = [\bba_k, \bbb_k]$, we pick up $j$-th dimension and choose a node $c_{k, j} \in (a_{k, j}, b_{k, j})$. Then $\mathrm{Q}_k$ is split into $\mathrm{Q}_k = \mathrm{Q}_k^{(1)} \cup \mathrm{Q}_k^{(2)}$ with 
\begin{equation}\label{split_subrectangle}
\begin{split}
& \mathrm{Q}_k^{(1)} = \prod_{i=1}^{j-1} [a_{k, i}, b_{k, i}] \times [a_{k, j}, c_{k, j}] \times \prod_{i=j+1}^d  [a_{k, i}, b_{k, i}], \\
& \mathrm{Q}_k^{(2)} = \prod_{i=1}^{j-1} [a_{k, i}, b_{k, i}] \times [c_{k, j}, b_{k,j}] \times \prod_{i=j+1}^d  [a_{k, i}, b_{k, i}].
\end{split}
\end{equation}
We can establish the partition $\mathcal{P}_{K+1}$ at level $K+1$ by
\begin{equation}
\mathcal{P}_{K+1} =  \left(\mathcal{P}_{K} \setminus \mathrm{Q}_k\right) \cup \mathrm{Q}_{k}^{(1)} \cup \mathrm{Q}_{k}^{(2)},
\end{equation}
and $\mathcal{X}_k = \mathcal{X}_k^{(1)}\cup \mathcal{X}_k^{(2)}$, $\mathcal{Y}_k = \mathcal{Y}_k^{(2)}\cup \mathcal{Y}_k^{(1)}$ with $P_k^{(i)} = |\mathcal{X}_k^{(i)}|$, $M_k^{(i)} = |\mathcal{Y}_k^{(i)}|$, respectively.

Although there are various choices of $c_{k, j}$, a rule of thumb for PA is to dig out the possible nodal points in the hyper-surfaces of the underlying integrand function. To this end, we can adopt the following strategy in practice. We can pick up the $j$-th dimension and $m$ equidistant points in the interval $[a_{k, j}, b_{k, j}]$:
\begin{equation}
c_{k, j} = a_{k, j} + \frac{l}{m}(b_{k,j} - a_{k, j}), \quad l = 1, \dots, m-1,
\end{equation} 
then select the node $c_{k, j}$ from $(m-1)\times d$ choices (typically, $m=2,4,8$) to attain the maximal value of the following {\bf difference gap}:
\begin{equation}\label{def.gap}
\Big | \frac{P_k^{(1)}}{P_k} - \frac{M_k^{(1)}}{M_k} \Big | + \Big | \frac{P_k^{(2)}}{P_k} - \frac{M_k^{(2)}}{M_k} \Big | = 2 \Big | \frac{P_k^{(1)}}{P_k} - \frac{M_k^{(1)}}{M_k} \Big |,
\end{equation}
which measures the gap between the distributions of positive and negative particles after splitting. Heuristically, a large difference gap means positive and negative particles tend to concentrate in different bins. That is, concentrations of positive or negative particles should be observed in opposite sides of the nodal hyper-surfaces.


\subsection{Deterministic bounds under arbitrary matching}

The clustering allows us to analyze the error function $\mathcal{E}(f)$ by bounding the local pieces $\mathcal{E}_k(f)$, which depend on both variations of test function $f$ and the bounds for star discrepancies of $\mathcal{X}_k$ and $\mathcal{Y}_k$ due to the celebrated Koksma-Hlawka inequality. 

In fact, the discrepancy theory plays a crucial role in high-dimensional numerical analysis and the most outstanding application is the low-discrepancy sequence in numerical integration, known as the quasi-Monte Carlo method \cite{Niederreiter1992,Caflisch1998}. Owing to the deterministic sequences with the star discrepancy of the order $(\log n)^d/n$, where $n$ is the required number of samples, the quasi-Monte Carlo achieves a theoretical convergence order of $n^{-1}$, 
but may still lose its effectiveness in higher dimension due to the factor $(\log n)^d$ \cite{Caflisch1998,TodorovDimovGeorgievaDimitrov2019}. Nonetheless, in certain applications it has been reported to obtain a satisfactory accuracy in evaluating the integrations up to $d= 360$ \cite{Paskov1996}.


Here we would like to focus on the potential application of discrepancy theory in clustering, without any specified construction of low-discrepancy sequences. It deserves to mention that the bounds of star discrepancies given in Eq.~\eqref{def.discrepancy_bound} are still of order $n^{-1}$ with $n$ the count of sequences, but their numerators are free from dimensionality $d$. As a consequence, in the deterministic bounds for SPADE as stated in Theorem \ref{thm.particle_annhilation_error_discrete_bound}, the dependance on $d$ is only embodied in the variation of test function, instead of the irregularity of points.

We begin to prove Theorem \ref{thm.particle_annhilation_error_discrete_bound}. 
Before that, we need Lemma~\ref{var_split} on the summation property and scaling-invariant property of variation in the sense of Hardy and Krause \cite{Owen2005}, as well as 
Lemma~\ref{lemma.removing_discrepancy} on the combinatorial property of the star discrepancy.


\begin{lemma}\label{var_split}
Let f be defined on the hyper-rectangle $\mathrm{Q} = [\bba, \bbb]$, 
$\{\mathrm{Q}_1, \dots, \mathrm{Q}_K\}$, $\mathrm{Q}_k = [\bba_k, \bbb_k], k=1, \dots, K$ be a binary partition of $\mathrm{Q}$. Then
\begin{equation}
V^{HK}_{[\bba, \bbb]}(f)  = \sum_{k=1}^{K} V^{HK}_{[\bba_k, \bbb_k]}(f).
\end{equation}
In addition, suppose $\phi$ is strictly monotone increasing function from $[\bba, \bbb]$ onto $[0,1]^d$, say, for $x_{i}^{(1)} < x_{i}^{(2)}$, $1\le i \le d$,
\[
\phi(x_1, \dots, x_{i-1}, x_{i}^{(1)}, x_i, \dots, x_d) < \phi(x_1, \dots, x_{i-1}, x_{i}^{(2)}, x_i, \dots, x_d). 
\]
 Let $\tilde{f}(\tilde{\bx}) = f(\bx)$ with $\tilde{\bx} = \phi(\bx)$, then $V^{HK}_{[\bba, \bbb]}(f) = V^{HK}_{[0,1]^d}(\tilde{f})$.
\end{lemma}
The proof of Lemma \ref{var_split} can be found in \cite{Owen2005}. Combining the Koksma-Hlawka inequality and Lemma \ref{var_split} yields the generalized Koksma-Hlawka inequality.

\begin{proposition}[Generalized Koksma-Hlawka inequality]\label{lemma_koksma_hlawka}
Let $f$ be of bounded variation on $\mathrm{Q} = [\bba, \bbb]$ in the sense of Hardy and Krause and $\phi$ is linear scaling from $\mathrm{Q} = [\bba, \bbb]$ to $[0,1]^d$. Then for any sequence $\mathcal{X} = (\bx_1, \dots, \bx_P) \subset \mathrm{Q}^P$, it has that
\begin{equation}
\Big | \frac{1}{P} \sum_{i=1}^P f(\bx_i) - \frac{1}{\lambda(\mathrm{Q})} \int_{\mathrm{Q}} f(\bx) \D \bx \Big | \le V_{[\bba, \bbb]}^{HK}(f) \cdot D_P^\ast(\phi(\bx_1), \dots, \phi(\bx_P)).
\end{equation}
\end{proposition}

\begin{proof}
Since the scaling $\phi$ in strictly monotone increasing, we have that 
\[
\begin{split}
\Big | \frac{1}{P} \sum_{i=1}^P f(\bx_i) - \frac{1}{\lambda(\mathrm{Q})} \int_{\mathrm{Q}} f(\bx) \D \bx \Big | & = \Big | \frac{1}{P} \sum_{i=1}^P \tilde{f}(\phi(\bx_i)) -  \int_{[0,1]^d} \tilde{f}(\tilde{\bx}) \D \tilde{\bx} \Big | \\
& \le V_{[0, 1]^d}^{HK}(\tilde{f}) \cdot D_P^\ast(\phi(\bx_1), \dots, \phi(\bx_P)) \\
& = V_{[\bba, \bbb]}^{HK}(f) \cdot D_P^\ast(\phi(\bx_1), \dots, \phi(\bx_P)),
\end{split}
\]
where the second inequality is the standard Koksma-Hlawka inequality, 
and the last one uses Lemma \ref{var_split}.
\end{proof}

In order to characterize the star discrepancy of a subsequence after removing $k$ points, we need a combinatorial lemma.

\begin{lemma}[Combinatorial property of star discrepancy]
\label{lemma.removing_discrepancy}
For any sequence $\mathcal{X} = (\bx_1, \dots, \bx_P) \subset [0, 1]^{d\times P}$ and arbitrary $k$ points $\mathcal{X}_k \subset \mathcal{X}$ in it, we have that
\begin{equation}
D_{P-k}^\ast(\{\bx_1, \dots, \bx_{P}\} \setminus \mathcal{X}_k) \le D_{P}^\ast(\bx_1, \dots, \bx_P) + \frac{k}{P}.
\end{equation}
\end{lemma}
\begin{proof}
Suppose $I$ is the critical box \cite{Niederreiter1972} that attains the star discrepancy of the sequence $\{\bx_1, \dots, \bx_{P}\} \setminus \mathcal{X}_k$, 
then
\[
\begin{split}
& D^\ast_{P-k}(\{\bx_1, \dots, \bx_{P}\} \setminus \mathcal{X}_k) - D^\ast_{P}(\bx_1, \dots, \bx_P) \\
& \le \Big | \frac{A(I, P-k, \mathcal{X} \setminus \mathcal{X}_k)}{P-k} - \lambda(I) \Big | - \Big | \frac{A(I, P, \mathcal{X})}{P} - \lambda(I) \Big | \\
& \le \Big | \frac{A(I, P-k, \mathcal{X} \setminus \mathcal{X}_k)}{P-k} - \frac{A(I, P, \mathcal{X})}{P} \Big |.
\end{split}
\]
Since  $0 \le A(I, P, \mathcal{X}) - A(I, P-k, \mathcal{X} \setminus \mathcal{X}_k) \le k$, it suffices to take $A(I, P,\mathcal{X}) - A(I, P-k, \mathcal{X} \setminus \mathcal{X}_k) = m$. Then 
\[
\Big | \frac{A(I, P-k, \mathcal{X} \setminus \mathcal{X}_k)}{P-k} - \frac{A(I, P, \mathcal{X})}{P} \Big | = \Big |\frac{kA(I, P-k, \mathcal{X} \setminus \mathcal{X}_k)}{P(P-k)} - \frac{m}{P} \Big | \le \frac{k}{P}.
\] 
The last inequality utilizes the fact that  
\[
0 \le \frac{kA(I, P-k, \mathcal{X} \setminus \mathcal{X}_k)}{P(P-k)} \le \frac{k}{P}
\]
and $m \le k$. The proof is completed. 
\end{proof}

Now we give a deterministic bound for $\xi_k$ in Eq.~\eqref{def.local_error_SPADE}, which also bounds the random variables in Section \ref{sec:matching}.

\begin{lemma}[A deterministic bound for random variables]
\label{thm_bound_random_variable}
Suppose $P_k \wedge M_k \le \gamma \sqrt{N}$ and $\osc_{[\bba_k,\bbb_k]}(f) < \infty$, then it has that
\begin{equation}\label{deterministic_bound_random_variable}
|\xi_k| \le \frac{\gamma }{4\sqrt{N}} \osc_{[\bba_k,\bbb_k]}(f).
\end{equation}
\end{lemma}
\begin{proof}
It suffices to consider $P_k \ge M_k$. From Eq.~\eqref{def.xi}, a direct calculation yields 
\[
\begin{split}
|\xi_k| & = \frac{1}{NP_k} \Big | P_k \sum_{i=1}^{M_k} f(\bx_{\sigma(i)}) - M_k \sum_{i=1}^{P_k} f(\bx_i) \Big | \\
&= \frac{1}{NP_k} \Big | (P_k-M_k) \sum_{i=1}^{M_k} f(\bx_{\sigma(i)}) - {M_k} \sum_{i=1}^{P_k} f(\bx_i) + {M_k}\sum_{i=1}^{M_k} f(\bx_{\sigma(i)}) \Big |. 
\end{split}
\]
Since
\[
(P_k-M_k) \sum_{i=1}^{M_k} f(\bx_{\sigma(i)}) = \sum_{i = 1}^{M_k} \sum_{j=1}^{P_k-M_k} f(\bx_{\sigma(i)}^{(j)})
\]
with $\bx_{\sigma(i)}^{(j)}$ being the copy of $\bx_{\sigma(i)}$, it has that
\[
|\xi_k| \le \frac{(P_k-M_k)M_k}{NP_k} \osc_{[\bba_k, \bbb_k]}(f) \le \frac{P_k}{4N} \osc_{[\bba_k, \bbb_k]}(f) \le \frac{\gamma }{4\sqrt{N}} \osc_{[\bba_k, \bbb_k]}(f).
\]
\end{proof}


With these preparations, we are able to finish the proof of Theorem \ref{thm.particle_annhilation_error_discrete_bound}. 

\begin{proof}[Proof of Theorem \ref{thm.particle_annhilation_error_discrete_bound}]
Suppose $P_k \ge M_k$ for the $k$-th cluster. Let $\mathcal{X}^A_k$ and $\mathcal{Y}_k^A$ denote the particles after removal, $|\mathcal{X}_k^A| = P_k - P_k \wedge M_k$ and $|\mathcal{Y}_k^A| = M_k - P_k \wedge M_k$. According to Proposition \ref{lemma_koksma_hlawka}, it has that
\[
\begin{split}
|\mathcal{E}_k(f)| & = \frac{M_k}{N} \Big | \frac{1}{M_k} \sum_{\mathcal{X}_k \setminus \mathcal{X}_k^A} f(\bx_i) - \frac{1}{M_k}\sum_{\mathcal{Y}_k} f(\by_i) \Big | \\
& \le \frac{M_k}{N} D^\ast_{M_k} ((\bx_1, \dots, \bx_{P_k}) \setminus \mathcal{X}_k^A) \cdot V_{[\bba_k, \bbb_k]}^{HK}(f)+ \frac{M_k}{N} D^\ast_{M_k}(\by_1, \dots, \by_{M_k})\cdot V_{[\bba_k, \bbb_k]}^{HK}(f).
\end{split}
\]

For the first term, when $P_k \ge \vartheta \sqrt{N}$, 
using Lemma \ref{lemma.removing_discrepancy} yields
\[
\begin{split}
\frac{M_k}{N} D^\ast_{M_k} ((\bx_1, \dots, \bx_{P_k}) \setminus \mathcal{X}_k^A) &\le \frac{M_k}{N} \frac{\vartheta \sqrt{N} + P_k - M_k }{P_k} \le \frac{(\vartheta \sqrt{N} + P_k)^2}{4P_k N} \\
& = \frac{P_k}{4N} + \frac{\vartheta^2}{4P_k} + \frac{\vartheta }{2\sqrt{N}} \le \frac{1}{\sqrt{N}}\left(\frac{\gamma}{4} + \frac{\vartheta}{2} + \frac{\vartheta^2}{4\gamma}\right),
\end{split}
\]
since the maximal value of $P_k/N + \vartheta^2/P_k$ for $P_k \in [\vartheta \sqrt{N}, \gamma \sqrt{N}]$ is attained at $P_k = \gamma \sqrt{N}$. 

When $M_k \le P_k  < \vartheta \sqrt{N}$, the bound of discrepancy becomes trivial since
\[
D^\ast_{M_k}((\bx_1, \dots, \bx_{P_k}) \setminus \mathcal{X}_k^A) \le 1 \le \frac{\vartheta \sqrt{N}}{P_k} \le \frac{\vartheta \sqrt{N}}{M_k}.
\]
In this case, we use the fact that
\[
\frac{M_k}{N} D^\ast_{M_k} ((\bx_1, \dots, \bx_{P_k}) \setminus \mathcal{X}_k^A)  \le \frac{M_k}{N} \le \frac{\vartheta}{\sqrt{N}} \le \frac{1}{\sqrt{N}}\left(\frac{\gamma}{4} + \frac{\vartheta}{2} + \frac{\vartheta^2}{4\gamma}\right).
\]

The second term is bounded by $\frac{\vartheta}{\sqrt{N}} V_{[\bba_k, \bbb_k]}^{HK}(f)$
due to Eq.~\eqref{def.discrepancy_bound}. 
Summing over $K$ terms arrives at Eq.~\eqref{def.HK_bounded_error_bound}. 


As for the bound in Eq.~\eqref{def.HK_bounded_error_bound_2}, according to Eq.~\eqref{def.local_error_SPADE}, it has that
\begin{equation}\label{formula_4.9}
\begin{split}
|\mathcal{E}_k(f)| \le & |\xi_k| + \frac{M_k}{N}\Big | \frac{1}{P_k} \sum_{i=1}^{P_k} f(\bx_i) - \frac{1}{\lambda(\mathrm{Q}_k)} \int_{\mathrm{Q}_k} f(\bx) \D \bx \Big | \\
&+ \frac{M_k}{N} \Big | \frac{1}{M_k} \sum_{i=1}^{M_k} f(\by_i) - \frac{1}{\lambda(\mathrm{Q}_k)} \int_{\mathrm{Q}_k} f(\bx) \D \bx \Big |\\
\le & \frac{\gamma}{4\sqrt{N}} \osc_{[\bba_k, \bbb_k]}(f) + \frac{M_k}{N} D_{P_k}^{\ast}(\phi_k(\bx_1), \dots, \phi_k(\bx_P))  \cdot V_{[\bba_k, \bbb_k]}^{HK}(f) \\
&+ \frac{M_k}{N} D_{M_k}^{\ast}(\phi_k(\by_1), \dots, \phi_k(\by_M)) \cdot V_{[\bba_k, \bbb_k]}^{HK}(f) \\
\le & \frac{\gamma}{4\sqrt{N}} \osc_{[\bba_k, \bbb_k]}(f) + \frac{2\vartheta}{\sqrt{N}} V_{[\bba_k, \bbb_k]}^{HK}(f),
\end{split}
\end{equation}
where the second inequality uses Proposition \ref{lemma_koksma_hlawka} and Lemma \ref{thm_bound_random_variable}. 
\end{proof}

A remarkable observation in the proof of Theorem \ref{thm.particle_annhilation_error_discrete_bound}  is that the worse-case upper bound of error function is attained when $M_k = (P_k + \vartheta \sqrt{N})/2$. Actually it accords with our intuition. If $M_k$ is sufficiently small, only a few particles are removed and errors should be small. When $M_k$ approaches to $P_k$, the bound of discrepancy for $\mathcal{Y}_k$ also controls the errors.  A bad situation may occur when $M_k$ is close to $P_k/2$.

\section{Matching and concentration inequalities}
\label{sec:matching}

The matching of two kinds of particles is the second issue we would like to address.  Intuitively, we can pair the particles that are `sufficiently close to each other' and resort to combinatorial algorithms for solving optimal assignment problems \cite{Kuhn1955}. However, the choice of metrics is rather subtle since the Euclidian distance is highly influenced by dimensionality $d$ and an embarrassing power factor $1/d$ seems to inevitably occur in error bounds (for instance, see \cite{Steinerberger2012}). Besides, we would like to extend our discussion to a general class of functions, such as that of bounded variation, and drop the assumption of continuity. For these purposes, it requires us to fully utilize the discrete nature of point distributions.

In this section, we would like to show that the random matching, in spite of its simple setting, proves to be surprisedly efficient and accurate. When $P_k \ge M_k$, a random matching from $\mathcal{Y}_k$ to $\mathcal{X}_k$ is equivalent to sampling $M_k$ times from a finite population $\mathcal{X}_k$ without replacement, and several powerful concentration inequalities will characterize the non-asymptotic bounds for the random variables $\xi_k$ in Eq.~\eqref{def.local_error_SPADE} and their independent sum $\sum_{k=1}^K \xi_k$. 

Roughly speaking, the concentration inequalities used in our analysis are divided into two categories. One is for sampling with replacement, initialized by the pioneering works of Hoeffding, Bennett and Bernstein \cite{Bennett1962,Hoeffding1963}, which characterizes the concentration behaviors of $\sum_{k=1}^K \xi_k$. The other category is for sampling without replacement, realized first by Serfling \cite{Serfling1974}, which bounds the random variables $\xi_k$ with a martingale structure. The latter has later been improved by Bardenet and Maillard \cite{BardenetMaillard2015}. The Talagrand-type inequality of sampling without replacement was given by Tolstikhin \cite{Tolstikhin2017}.

There are still a vast of researches on concentration inequalities and a good survey can be found in \cite{bk:BoucheronLugosiMassart2013}.  An exhaustion of them is definitely beyond the scope of this article. Here we only use several typical concentration inequalities, such as those of Hoeffding's type and Bernstein's type. The first part mainly discusses Bernstein's inequality and the Serfling-Hoeffding inequality. The second part concentrates on the Serfling-Bernstein inequality. Again, the discrepancy theory will be used for bounding the local variance in sampling without replacement.

\subsection{Stochastic bounds under random matching}

We first investigate the Serfling-Hoeffding concentration inequalities for sampling without replacement. 
Without loss of generality, we consider the situation that the number of positive particles 
is no less than that of negative particles.


Consider the partial summation
\begin{equation}
S_k = \sum_{t=1}^k f(X_t),
\end{equation}
where $X_t$ are randomly sampled from a finite population $\mathcal{X} = (\bx_1, \dots, \bx_P)$ without replacement as defined in Eq.~\eqref{probability_law_sampling_without_replacement}. It is easy to verify that
\begin{equation}
\frac{\mathbb{E} S_k}{k} - \frac{S_P}{P} = \mathbb{E}  \left(  \frac{1}{k} \sum_{t=1}^k  f(X_t) -  \frac{1}{P} \sum_{i=1}^P f(\bx_i) \right) = 0,
\end{equation}
and the two-side tail probabilities $U_M(\delta)$ and $V_M(\delta)$
\begin{align}
& U_M(\delta) = \Pr \left\{  \max_{M \le k \le P-1} \Big | \frac{S_k}{k} - \frac{1}{P} \sum_{i=1}^P f(\bx_i) \Big | \ge \delta \right\}, \\
& V_M(\delta) = \Pr \left\{  \max_{1 \le k < M} \Big | \frac{S_k}{P - k} - \frac{1}{P} \sum_{i=1}^P f(\bx_i) \Big | \ge \frac{M \delta}{P-M} \right\}
\end{align}
are bounded by the Serfling-Hoeffding inequality, which characterizes the deviation between $S_k/k$ and the arithmetic mean of $\mathcal{X}$. Here we only list the concentration bounds in Lemma \ref{lemma_concentration_HS_ineq} and the detailed proof can be found in \cite{BardenetMaillard2015}.



\begin{lemma}[The Hoeffding-Serfling inequality]
\label{lemma_concentration_HS_ineq} 
For $\mathcal{X}= (\bx_1, \dots, \bx_P) \subset \mathrm{Q}^P$, $\mathrm{Q} = [\bba, \bbb]$, let $(X_1, \dots, X_M)$ be a list of a flexible size $M < P$ sampled without replacement from $\mathcal{X}$. Then
for any function $f$ with bounded oscillator norm $\osc_{[\bba, \bbb]}(f) < \infty$ and $\delta > 0$, we have
\begin{equation}\label{ineq_concentration_HS_ineq}
\begin{split}
& U_M(\delta) \le 2 \exp \left(\frac{-2M \delta^2}{(1-{M}/{P})(1+{1}/{M})({\osc_{[\bba, \bbb]}(f)})^2}\right), \\
& V_M(\delta) \le 2\exp \left(\frac{-2M \delta^2}{(1 - (M-1)/P)({\osc_{[\bba, \bbb]}(f)})^2}\right).
\end{split}
\end{equation} 
Moreover, 
for any $\varepsilon \in (0, 1)$, it holds with probability higher than $1 - \varepsilon$ that
\begin{equation}\label{ineq_concentration_HS_ineq1}
\Big | \frac{1}{M} \sum_{t=1}^M f(X_t) - \frac{1}{P} \sum_{i=1}^P f(\bx_i) \Big | \le \osc_{[\bba, \bbb]}(f) \sqrt{\frac{\rho_M \log(2/\varepsilon)}{2M}},
\end{equation}
where  
\begin{equation}\label{def.rho_k}
\rho_M = 
\left\{
\begin{split}
& \left(1 -  \frac{M-1}{P}\right), \quad &M < P/2, \\
& \left(1 - \frac{M}{P}\right) \left(1+\frac{1}{M}\right), \quad &M \ge P/2.
\end{split}
\right.
\end{equation}
Note in passing that Eq.~\eqref{ineq_concentration_HS_ineq1} holds for $P=M$, too. 
\end{lemma} 

The factor $M/P$ in Eq.~\eqref{def.rho_k} for $\rho_M$ reflects the influence of sampling fraction \cite{Serfling1974} and leads to stronger concentration compared to independent sampling \cite{Tolstikhin2017}.

We would like to use Lemma \ref{lemma_concentration_HS_ineq} to bound the random variable
\begin{equation}\label{random_variable}
\xi = \frac{M}{N} \left( \frac{1}{M}\sum_{t=1}^M f(X_t) - \frac{1}{P}\sum_{i=1}^P f(\bx_i)\right),
\end{equation} 
which corresponds to $\xi_k$ in the local error term \eqref{def.local_error_SPADE}.  A key observation is that $\sqrt{M \rho_M} = 0$ when $M= 0$ or $M= P$, 
and $\sqrt{M \rho_M}$ attains it maximal value at $M = \frac{P+1}{2}$. 
This coincides with our intuition since $M = 0$ means no particle is annihilated, 
while $M = P$ corresponds to the case in which all positive particles are selected from the population, so that errors induced by random matching vanish.

\begin{proposition}[The first stochastic bound for random variables]
\label{thm.transportation_cost_HS_bound}
Suppose $P \ge M$, then any $\varepsilon \in (0, 1)$, it holds with probability higher than $1-\varepsilon$ that
\begin{equation}\label{def.transportation_cost_Hoeffding}
|\xi| \le \sqrt{\frac{\log(2/\varepsilon)}{8}}  \left(\frac{P+1}{N\sqrt{P}}\right) \osc_{[\bba, \bbb]}(f)
\end{equation}
for any function $f$ with $\osc_{[\bba, \bbb]}(f) < \infty$. In particular, when $P \le \gamma \sqrt{N}$, it has that
\begin{equation}\label{def.transportation_cost_Hoeffding_2}
|\xi| \le \sqrt{\frac{\log(2/\varepsilon)}{8}} \left(\gamma^{1/2} + \frac{1}{\gamma^{1/2} N^{1/2}}\right) \frac{\osc_{[\bba, \bbb]}(f)}{N^{3/4}}. 
\end{equation}
\end{proposition}


\begin{proof}
It can be readily verified from Eq.~\eqref{def.rho_k} that
\begin{equation}\label{ieq:rho}
\frac{M}{N}\sqrt{\frac{\rho_M}{M}} \le \frac{P+1}{2N\sqrt{P}}.
\end{equation}

Applying the concentration inequality \eqref{ineq_concentration_HS_ineq1} into Eq.~\eqref{random_variable} and using Eq.~\eqref{ieq:rho} gives Eq.~\eqref{def.transportation_cost_Hoeffding}.

Inserting
\[
\sqrt{P} + \frac{1}{\sqrt{P}} \le \gamma^{1/2} N^{1/4} + \frac{1}{\gamma^{1/2} N^{1/4}},
\quad 1 \le P \le \gamma\sqrt{N}
\]
into Eq.~\eqref{def.transportation_cost_Hoeffding} yields Eq.~\eqref{def.transportation_cost_Hoeffding_2}.
\end{proof}

The variance estimation of the random matching is also obtained by Lemma \ref{lemma_concentration_HS_ineq}. In contrast to a trivial bound obtained from Eq.~\eqref{deterministic_bound_random_variable}:
\begin{equation}\label{trivial}
\mathbb{E}\xi^2 \le \frac{\gamma^2}{16} \frac{\left(\osc_{[\bba, \bbb]}(f)\right)^2}{N},
\end{equation}
we would like to show that  the second moment $\mathbb{E}\xi^2$ gains an extra order of $N^{-1/2}$ with the concentration bound in Eq.~\eqref{ineq_concentration_HS_ineq}.

\begin{proposition}[The first bound for the variance of random matching]
\label{thm_first_stochastic_bound_second_moment}
Suppose $P\ge M>0$, $\osc_{[\bba, \bbb]}(f) < \infty$ and there exists a constant $\gamma< 1$ such that $P < \gamma\sqrt{N}$, then
\begin{equation}
\mathbb{E}\xi^2 \le H_1(N, \gamma) \left(\frac{\osc_{[\bba, \bbb]}(f)}{N^{3/4}} \right)^2,
\end{equation}
where $H_1(N, \gamma)$ is given in Eq.~\eqref{random_variance_bound}.
\end{proposition}

\begin{proof}
It starts from
\[
\mathbb{E} \xi^2= \frac{2 M^2}{N^2} \int_0^{+\infty} u \Pr\left\{ \Big |\frac{S_M}{M} - \frac{1}{P}\sum_{i=1}^P f(\bx_i) \Big | \ge u \right\} \D u.
\]
Owing to the concentration bound \eqref{ineq_concentration_HS_ineq}, it further has that 
\[
\begin{split}
\mathbb{E}\xi^2 &\le \frac{4M^2}{N^2}\int_0^{+\infty} u \exp\left( -\frac{2M u^2}{(1-{M}/{P})(1+{1}/{M})(\osc_{[\bba, \bbb]}(f))^2}\right) \D u \\
& = \frac{(P-M)(1+M)}{P} \frac{\left(\osc_{[\bba, \bbb]}(f)\right)^2}{N^2} \\
& \le H_1(N, \gamma) \frac{\left(\osc_{[\bba, \bbb]}(f)\right)^2}{N^{3/2}}. 
\end{split}
\]
Since $\mathbb{E} \xi = 0$, it completes the proof. 
\end{proof}

In order to characterize the concentration of the independent sum $\sum_{k=1}^K \xi_k$, we need a basic version of  Bernstein's inequality.

\begin{lemma}[Bernstein's inequality]
\label{basic_berstein_ineq}
Suppose $\xi_k$ are independent random variables with $|\xi_k| \le b$, $1\le k \le K$, and $v = \sum_{k=1}^K \mathbb{E} \xi_k^2$. Then for any $\varepsilon \in (0, 1)$, it holds with probability higher than $1 - \varepsilon$ that
\begin{equation}\label{Bernstein_bound}
\Big | \sum_{k=1}^K \xi_k \Big | \le \sqrt{2v\log(2/\varepsilon)} + \frac{2 b \log(2/\varepsilon)}{3}.
\end{equation}
\end{lemma}

Combining Proposition \ref{lemma_koksma_hlawka}, Proposition \ref{thm.transportation_cost_HS_bound}, Proposition \ref{thm_first_stochastic_bound_second_moment} and Lemma \ref{basic_berstein_ineq} together finishes the proof of Theorem \ref{thm.particle_annhilation_error_random_bound_no_variance}.

\begin{proof}[Proof of Theorem \ref{thm.particle_annhilation_error_random_bound_no_variance}]
Since $\mathbb{E}\xi_k = 0$, it has that
\[
|\mathbb{E}\mathcal{E}_k(f) | = \Big | \frac{1}{N} \frac{M_k}{P_k} \sum_{i=1}^{P_k} f(\bx_i) - \frac{1}{N}\frac{M_k}{M_k} \sum_{i=1}^{M_k} f(\by_i) \Big | \le  \frac{2\vartheta}{\sqrt{N}} V_{[\bba_k, \bbb_k]}^{HK}(f),
\]
the proof of which is the same as that in Eq.~\eqref{formula_4.9}. Combining the above estimate with $| \mathbb{E} \mathcal{E}(f)  | \le \sum_{k=1}^K  | \mathbb{E} \mathcal{E}_k(f)  |$ recovers Eq.~\eqref{random_expection_bound}. 

Since $\mathcal{E}_k(f)$ are mutually independent, it has 
\begin{equation}\label{var_decomp}
\textup{Var}(\mathcal{E}(f)) = \sum_{k=1}^K \textup{Var}(\mathcal{E}_k(f)).
\end{equation}
Thus Proposition \ref{thm_first_stochastic_bound_second_moment} implies Eq.~\eqref{random_variance_bound} in view of Eq.~\eqref{def.variance_random_matching} . 

Finally, setting $v = \textup{Var}(\mathcal{E}(f))$ and $b = \max_k \frac{\gamma}{4\sqrt{N}} \osc_{[\bba_k, \bbb_k]}(f)$ in Bernstein's inequality \eqref{Bernstein_bound} directly yields Eq.~\eqref{random_error_Bernstein_bound}.
\end{proof}


\subsection{Variance estimation via bounding discrepancies}

We first illustrate concentration bounds of Bernstein's type for the tail probabilities $U_M(\delta)$ and $V_M(\delta)$, see Lemma~\ref{lemma_concentration_BS_ineq}, which 
uses the local variance
\begin{equation}
\sigma^2 = \frac{1}{P} \sum_{i=1}^P \left( f(\bx_i) - \frac{S_P}{P}\right)^2.
\end{equation}
Its detailed proof is omitted and the interested readers can refer to \cite{BardenetMaillard2015}. 
Afterwards we will complete the proof of Theorem \ref{thm.particle_annhilation_error_random_bound_variance}, using the bounds of discrepancies to estimate the local variances and the concentration bounds to estimate the second moment $\mathbb{E} \xi_k^2$.


\begin{lemma}[The Bernstein-Serfling inequality]
 \label{lemma_concentration_BS_ineq}
For $\mathcal{X}= (\bx_1, \dots, \bx_P) \subset \mathrm{Q}^P$, $\mathrm{Q} = [\bba, \bbb]$, let $(X_1, \dots, X_M)$ be a list of a flexible size $M < P$ sampled without replacement from $\mathcal{X}$. Then for any $\delta>0$ and any $\varepsilon \in (0, 1)$, we have
\begin{equation}\label{ineq_concentration_BS_ineq}
\begin{split}
&U_M(\delta) \le 2\exp\left( \frac{- M \delta^2/2}{\frac{P-M}{P}(\frac{M+1}{M}\sigma + \frac{\sqrt{\log(2/\varepsilon^2)(P-M-1)}}{M}{\osc_{[\bba, \bbb]}(f)} )\sigma +\frac{2}{3}\delta \osc_{[\bba, \bbb]}(f)}\right) + 2 \varepsilon,\\
&V_M(\delta) \le 2\exp\left( \frac{- M \delta^2/2}{(\frac{P-M+1}{P}\sigma +  \frac{\sqrt{\log (2/\varepsilon^2)(M-1)}}{P} {\osc_{[\bba, \bbb]}(f)})\sigma  + \frac{2}{3}\delta \osc_{[\bba, \bbb]}(f) }\right) + 2 \varepsilon.
\end{split}
\end{equation}
Moreover, it holds with probability higher than $1-\varepsilon$ that
\begin{equation}\label{ieq:rhok}
\Big | \frac{1}{M} \sum_{t=1}^M f(X_t) - \frac{1}{P} \sum_{i=1}^P f(\bx_i) \Big | \le 2\sigma \sqrt{\frac{\rho_M \log(2/\varepsilon)}{M}} + 2\osc_{[\bba, \bbb]}(f) \left(\frac{\kappa_M \log(2/\varepsilon)}{M}\right), 
\end{equation}
where $\rho_M$ is given in Eq.~\eqref{def.rho_k} and 
\begin{equation}
\kappa_M = 
\left\{
\begin{split}
& \frac{4}{3} + \sqrt{\frac{M}{P} \left(\frac{M-1}{P-M+1}\right)}, \quad &M < P/2, \\
& \frac{4}{3} + \sqrt{\left(\frac{P-M-1}{M+1}\right) \left(\frac{P-M}{P}\right)}, \quad &M \ge P/2.
\end{split}
\right.
\end{equation}
\end{lemma}


When $\osc_{[\bba, \bbb]}(f) < \infty$, a trivial bound of the local variance is
\begin{equation}
\sigma^2 = \frac{1}{P}\sum_{i=1}^P \left(\sum_{j=1}^P\frac{f(\bx_i) - f(\bx_j)}{P}\right)^2 \le \left(\osc_{[\bba, \bbb]}(f)\right)^2.
\end{equation}
For the points $(\bx_1, \dots, \bx_P)$ with bounded star discrepancy,
we could provide another deterministic bound for $\sigma^2$.


\begin{theorem}[A deterministic bound for local variance]
\label{thm.discrepancy_bound_variance}
Suppose $\mathcal{X} = (\bx_1, \dots, \bx_P) \subset \mathrm{Q}^P$ with $\mathrm{Q} = [\bba, \bbb]$ and there is a constant $\delta < P$ such that $D^\ast_P(\phi(\bx_1), \dots, \phi(\bx_P)) < \delta/P$, where $\phi$ is a scaling from $[\bba, \bbb]$ to $[0, 1]^d$. Then we have 
\begin{equation}\label{variance_estimation}
\sigma^2 \le \frac{\delta V_{[\bba, \bbb]}^{HK}(f^2)}{P} + \frac{\delta V_{[\bba, \bbb]}^{HK}(f) \cdot \sup_{[\bba, \bbb]}|f|}{P} + \Vert f \Vert_{L^2([\bba, \bbb])} \cdot \osc_{[\bba, \bbb]}(f),
\end{equation}
provided $\osc_{[\bba, \bbb]}(f) < \infty$, $V_{[\bba, \bbb]}^{HK}(f) < \infty$ and $V_{[\bba, \bbb]}^{HK}(f^2) < \infty$.

\end{theorem}
\begin{proof}
Since $\sum_{i=1}^P f(\bx_i) = P S_P$, a direct calculation yields that 
\[
\sigma^2 = \frac{1}{P}\sum_{i=1}^P f^2(\bx_i) - \left(\frac{S_P}{P}\right)^2,
\]
and then
\[
\begin{split}
\sigma^2  &= \frac{1}{P}\sum_{i=1}^P f^2(\bx_i)  - \frac{1}{\lambda(\mathrm{Q})}\int_{\mathrm{Q}} f^2(\bx) \D \bx + \frac{1}{\lambda(\mathrm{Q})} \int_{\mathrm{Q}} f^2(\bx) \D \bx - \left(\frac{S_P}{P}\right)^2 \\
&\le  V_{[\bba, \bbb]}^{HK}(f^2) \cdot D^\ast_P(\phi(\bx_1), \dots, \phi(\bx_P)) + \frac{1}{\lambda(\mathrm{Q})} \int_{\mathrm{Q}} f^2(\bx) \D \bx - \left(\frac{S_P}{P}\right)^2.
\end{split}
\]

The bounded star discrepancy in the first term gives the first part of the bound in 
Eq.~\eqref{variance_estimation}. For the second term, we have that
\[
\begin{split}
\frac{1}{\lambda(\mathrm{Q})} \int_{\mathrm{Q}} f^2(\bx) \D \bx  = & \frac{1}{\lambda(\mathrm{Q})} \int_{\mathrm{Q}} f(\bx)\left(f(\bx) - \frac{S_P}{P}\right) \D \bx \\
&+ \frac{S_P}{P} \left(\frac{1}{\lambda(\mathrm{Q})}\int_{\mathrm{Q}} f(\bx) \D \bx - \frac{S_P}{P}\right) +\left(\frac{S_P}{P}\right)^2\\
\le & \left(\frac{1}{\lambda(\mathrm{Q})} \int_{\mathrm{Q}} f^2(\bx) \D \bx\right)^{1/2}  \left(\frac{1}{\lambda(\mathrm{Q})} \int_{\mathrm{Q}} \left(f(\bx) - \frac{S_P}{P}\right)^2 \D \bx\right)^{1/2} \\
& + \frac{\delta V_{[\bba, \bbb]}^{HK}(f) \cdot \sup_{[\bba, \bbb]}|f|}{P} + \left(\frac{S_P}{P}\right)^2\\
\le &\Vert f \Vert_{L^2([\bba, \bbb])} \cdot \osc_{[\bba, \bbb]}(f)  + \frac{\delta V_{[\bba, \bbb]}^{HK}(f) \cdot \sup_{[\bba, \bbb]}|f| }{P} + \left(\frac{S_P}{P}\right)^2,
\end{split}
\]
which recovers the last two parts of the bound in Eq.~\eqref{variance_estimation}.  
\end{proof}

Combining Theorem \ref{thm.discrepancy_bound_variance} 
with
Lemma \ref{lemma_concentration_BS_ineq} produces 
the second stochastic bound for the random variable $\xi$
in Eq.~\eqref{random_variable}.

\begin{proposition}[The second stochastic bound for random variables]
\label{thm.transportation_cost_BS_bound}
Suppose $P\ge M > 0$, then for any $\varepsilon \in (0, 1)$, 
it holds with probability higher than $1- \varepsilon$ that
\begin{equation}\label{def.transportation_cost_Bernstein}
|\xi | \le   \sqrt{\log(2/\varepsilon)} \left( \frac{P+1}{N\sqrt{P}}\right) \sigma+  \log(2/\varepsilon) \left(\frac{8+3\sqrt{2}}{3N} \right) \osc_{[\bba, \bbb]}(f),
\end{equation}
provided $\osc_{[\bba, \bbb]}(f) < \infty$, $V_{[\bba, \bbb]}^{HK}(f) < \infty$ and $V_{[\bba, \bbb]}^{HK}(f^2) < \infty$. In particular, when the inequality \eqref{variance_estimation} holds for $\delta = \vartheta \sqrt{N}$ and $P \le \gamma \sqrt{N}$, it holds with probability higher than $1- \varepsilon$ that
\begin{equation}\label{def.transportation_cost_Bernstein_2}
|\xi| \le \frac{3\sqrt{\log(2/\varepsilon)}}{2N^{3/4}}  \mathcal{V}_{[\bba, \bbb]}(f, \vartheta, \gamma)  + \frac{(8+3\sqrt{2})\log(2/\varepsilon)}{3N}  \osc_{[\bba, \bbb]}(f),
\end{equation}
where 
\begin{equation}\label{def.V}
\mathcal{V}_{[\bba, \bbb]}(f, \vartheta, \gamma) = \sqrt{\vartheta V_{[\bba, \bbb]}^{HK}(f^2)  + \vartheta V_{[\bba, \bbb]}^{HK}(f) \cdot \sup_{[\bba, \bbb]}|f| + \gamma\Vert f \Vert_{L^2([\bba, \bbb])} \cdot \osc_{[\bba, \bbb]}(f)}.
\end{equation}
\end{proposition}

\begin{proof}
When $P=1$, we have $\kappa_M = 4/3$ and  $\rho_M = 0$,
with which Eq.~\eqref{def.transportation_cost_Bernstein} holds according to Lemma \ref{lemma_concentration_BS_ineq}. Below we only need to consider $P >1$.

The first part of the bound in Eq.~\eqref{def.transportation_cost_Bernstein}  
can be readily obtained via applying the fact \eqref{ieq:rho} into the first part of the bound 
in Eq.~\eqref{ieq:rhok}.

We have the second part of the bound in Eq.~\eqref{def.transportation_cost_Bernstein}  
from the second part of the bound 
in Eq.~\eqref{ieq:rhok} if noting  
\begin{equation}
\kappa_M \le \frac{4}{3} + \sqrt{\frac{P-2}{2(P+2)}} \le \frac{4}{3} + \frac{\sqrt{2}}{2},
\end{equation}
which can be verified using the fact that the monotonically increasing function $g(x) = \frac{x}{P} \left(\frac{x-1}{P-x+1}\right)$ for $x\in[1,P/2]$
and the monotonically decreasing function $h(x) =  \left(\frac{P-x-1}{x+1}\right) \left(\frac{P-x}{P}\right)$
for $x\in[P/2, P]$ both attain the maximal value at $x = P/2$.

Implementing $\delta = \vartheta \sqrt{N}$ and $P \le \gamma \sqrt{N}$
in Eq.~\eqref{variance_estimation} yields $\sigma\le \frac{N^{1/4}}{\sqrt{P}} \mathcal{V}_{[\bba, \bbb]}(f, \vartheta, \gamma)$,
and substituting it into Eq.~\eqref{def.transportation_cost_Bernstein} reaches Eq.~\eqref{def.transportation_cost_Bernstein_2}.
\end{proof}


\begin{proposition}[The second bound for variance of random matching]
\label{thm_deterministic_bound_second_moment}
Suppose $P\ge M>0$, $\osc_{[\bba, \bbb]}(f) < \infty$ and there exists a constant $\gamma< 1$ such that $P < \gamma\sqrt{N}$. Then
\begin{equation}\label{second_bound_variance}
\mathbb{E}\xi^2 \le \frac{H_2(N)}{N^{1/2}(N+1)}  \left( \mathcal{V}_{[\bba, \bbb]}(f, \vartheta, \gamma) \right)^2 + \frac{H_3(\gamma, N)}{N(N+1)}\left(\osc_{[\bba, \bbb]}(f)\right)^2,
\end{equation}
where $H_2(N)$ and $H_3(\gamma, N)$ are given in Eq.~\eqref{random_variance_bound_disc}.
\end{proposition}

\begin{proof}
As a convenience, we denote the bound in Eq.~\eqref{def.transportation_cost_Bernstein_2} by
\[
B_{\gamma, \vartheta, N}(f) =  \frac{C_1}{N^{3/4}}  \mathcal{V}_{[\bba, \bbb]}(f, \vartheta, \gamma)  + \frac{C_2}{N}  \osc_{[\bba, \bbb]}(f),
\]
where the coefficients $C_1$ and $C_2$ both depends on only $\varepsilon$.
Then $\Pr(\{ |\xi|  > B_{\gamma, \vartheta, N}(f)\}) <  \varepsilon$,
and 
\[
\mathbb{E} \left( \xi^2 \cdot \mone_{\{ |\xi| \le B_{\gamma, \vartheta, N}(f)  \}} \right) \\
 \le (B_{\gamma, \vartheta, N}(f))^2 (1 - \Pr(\{ |\xi|  > B_{\gamma, \vartheta, N}(f)  \})).
\]
On the other hand,  
the trivial bound in Eq.~\eqref{trivial} implies 
\[
\mathbb{E} \left( \xi^2 \cdot \mone_{\{ |\xi| > B_{\gamma, \vartheta, N}(f)  \}} \right) \le \frac{\gamma^2}{16N} \left(\osc_{[\bba, \bbb]}(f)\right)^2  \Pr(\{ |\xi|  > B_{\gamma, \vartheta, N}(f)  \}).
\]
Together, it has for $B_{\gamma, \vartheta, N}(f) \le \frac{\gamma}{4\sqrt{N}} \osc_{[\bba, \bbb]}(f)$ that 
\[
\begin{split}
\mathbb{E} \xi^2 = & \mathbb{E} \left( \xi^2 \cdot \mone_{\{ |\xi| \le B_{\gamma, \vartheta, N}(f)  \}} \right) + \mathbb{E} \left(\xi^2 \cdot \mone_{\{ |\xi| > B_{\gamma, \vartheta, N}(f)  \}} \right)\\
\le &\left[ \frac{\gamma^2}{16N} \left(\osc_{[\bba, \bbb]}(f)\right)^2 - (B_{\gamma, \vartheta, N}(f))^2\right] \varepsilon + (B_{\gamma, \vartheta, N}(f))^2\\
= &(B_{\gamma, \vartheta, N}(f))^2(1-\varepsilon) + \frac{\gamma^2 \varepsilon}{16N} \left(\osc_{[\bba, \bbb]}(f)\right)^2 \\
\le & \frac{2C_1^2}{N^{3/2}} \left( \mathcal{V}_{[\bba, \bbb]}(f, \vartheta, \gamma) \right)^2 (1 - \varepsilon) + \frac{2C_2^2}{N^2} \left(\osc_{[\bba, \bbb]}(f)\right)^2 (1 -\varepsilon) + \frac{\gamma^2 \varepsilon}{16N} \left(\osc_{[\bba, \bbb]}(f)\right)^2,
\end{split}
\]
which also holds for $B_{\gamma, \vartheta, N}(f) > \frac{\gamma}{4\sqrt{N}} \osc_{[\bba, \bbb]}(f)$ due to Eq.~\eqref{trivial}. 

Finally, taking $\varepsilon = \frac{1}{N+1}$ recovers Eq.~\eqref{second_bound_variance}. 
\end{proof}

The remaining task is to complete the proof of Theorem \ref{thm.particle_annhilation_error_random_bound_variance}.

\begin{proof}[Proof of Theorem \ref{thm.particle_annhilation_error_random_bound_variance}]
According to Eqs.~\eqref{def.variance_random_matching}, \eqref{var_decomp} and \eqref{second_bound_variance}, the variance $\textup{Var}(\mathcal{E}(f))$ is bounded by
\[
\textup{Var}(\mathcal{E}(f)) \le \frac{H_2(N)}{N^{1/2}(N+1)}  \sum_{k=1}^K \left( \mathcal{V}_{[\bba_k, \bbb_k]}(f, \vartheta, \gamma) \right)^2 + \frac{H_3(\gamma, N)}{N(N+1)} \sum_{k=1}^K \left(\osc_{[\bba_k, \bbb_k]}(f)\right)^2.
\]
For the first term, we have 
\[
\begin{split}
\sum_{k=1}^K \left( \mathcal{V}_{[\bba_k, \bbb_k]}(f, \vartheta, \gamma) \right)^2 = &\sum_{k=1}^K \vartheta V_{[\bba_k, \bbb_k]}^{HK}(f^2) + \sum_{k=1}^K  \vartheta V_{[\bba_k, \bbb_k]}^{HK}(f)  \sup_{[\bba_k, \bbb_k]}|f| \\
&+ \sum_{k=1}^K \gamma\Vert f \Vert_{L^2([\bba_k, \bbb_k])} \osc_{[\bba_k, \bbb_k]}(f) \\
\le & \vartheta V_{[\bba, \bbb]}^{HK}(f^2) + \vartheta V_{[\bba, \bbb]}^{HK}(f)  \sup_{[\bba, \bbb]}|f| + \gamma \sum_{k=1}^K \Vert f \Vert_{L^2([\bba_k, \bbb_k])}  \osc_{[\bba_k, \bbb_k]}(f).
\end{split}
\]
As a result, it recovers Eq.~\eqref{random_variance_bound_disc}.
\end{proof}


\section{Numerical validation}
\label{sec:numerical}

To illustrate the negative particle weights, 
we consider the numerical integration of a determinantal function
\begin{equation}
I(f) = \int_{\mathbb{R}^d} f(\bv) g(\bv) \D \bv, \quad g(\bv) = \frac{\textup{det}(G(\bv))}{\int_{\mathbb{R}^d} \textup{det}(G(\bv)) \D \bv},
\end{equation}
the prototypes of which are the Slater-determinant-type wave functions in quantum physics and the determinantal point process,
where $G(\bv)=(G_{ij}(\bv))_{m\times m}$ is a $m\times m$ matrix-valued function 
with $G_{ij}(\bv)$ being a function of $d$ variables. Since $g(\bv)$ is usually not positive semi-definite, the standard Monte Carlo approach is to adopt $|g(\bv)|$ as the unnormalized instrumental probability density and $w(\bv) = g(\bv)/|g(\bv)| \in \{-1, 1\}$ as the importance weight as follows
\[
I(f) = \mathcal{Z}  \int_{\mathbb{R}^d} f(\bv) w(\bv) \frac{|g(\bv)|}{\mathcal{Z}} \D \bv, \quad \mathcal{Z} = \int_{\mathbb{R}^d} |g(\bv)| \D \bv.
\]
It yields two kinds of estimators $I_{+}(f)$ and $I_{-}(f)$ \cite{bk:RobertCasella2004,TroyerWiese2005}: 
\begin{equation} \label{eq.signed_estimator_1}
I_{+}(f) =  \mathcal{Z}  \left(\frac{\sum_{i=1}^P f(\bx_i) w(\bx_i) + \sum_{i=1}^M f(\by_i) w(\by_i) }{P + M}\right),
\end{equation}
and
\begin{equation}\label{eq.signed_estimator_2}
I_{-}(f)  =  \frac{\sum_{i=1}^P f(\bx_i) w(\bx_i) + \sum_{i=1}^M f(\by_i) w(\by_i) }{ \sum_{i=1}^P w(\bx_i) + \sum_{i=1}^M w(\by_i) } = \frac{\sum_{i=1}^P f(\bx_i) - \sum_{i=1}^M f(\by_i)}{P - M},
\end{equation}
where $w(\bx_i) = 1$ and $w(\by_i) = -1$. The latter estimator $I_{-}(f)$ utilizes the strong law of large number 
\[
\mathcal{Z}  \left(\frac{\sum_{i=1}^P w(\bx_i) + \sum_{i=1}^M w(\by_i)}{P+M} \right) \to 1 \quad \textup{as} \quad P+M \to +\infty.
\]
Both estimators have the form of Eq.~\eqref{def.empirical_signed_measure}. 
In particular, $I_{-}(f)$ avoids calculating the normalizing constant $\mathcal{Z}$.

\subsection{Experiment setup}
 We adopt the convention that $\bv=(v_1, \dots, v_d)=(\bv_1, \dots, \bv_m)$ with $n = d/m$, $\bv_j = (v_{(j-1)n+1}, \dots, v_{jn})\in \mathbb{R}^{n}$, and set
\begin{equation}
G_{ij}(\bv) =
\left\{
\begin{split}
&\phi_i(\bv_i), \quad & i = j, \\
&\epsilon \phi_i(\bv_j) \quad & i \ne j,
\end{split}
\right.
\end{equation}
where $\epsilon \in (0, 1)$ is a parameter and $\phi_i(\bv_j)~(1\le i \le m, 1 \le j \le m)$ are given by
\begin{equation}
\phi_i(\bv_j) = \frac{1}{(2\pi)^{d/2m}} \exp\left(-\frac{|\bv_j - \bar{\bv}_i|^2}{2}\right).
\end{equation}
The central positions $\bar{\bv}_i\in \mathbb{R}^{n}$ are composed of $n/3$ integer-valued grid points chosen from the rectangular box $[-4, 5] \times [-2, 3] \times [-2, 3]$.

The positive and negative particles come from the estimator $I_{-}(f)$ in Eq.~\eqref{eq.signed_estimator_2}. The relative error 
\begin{equation}
\textup{r.e.}(f) = \frac{|\mathcal{E}(f)|}{|\int f \D \mu|}
\end{equation}
are adopted  to evaluate the performance.
Below we consider five test functions,  
denoted by $f_1, f_2, f_3, f_4, f_5$. 
According to Eq.~\eqref{variation_continuous_mixed_derviate}, 
the explicit formula for the variation of $f_1$, $f_2$, $f_5$ on $[0,1]^d$ in the sense of Hardy and Krause can be obtained. One can see that $V_{[0,1]^d}^{HK}(f_1)$ and $V_{[0,1]^d}^{HK}(f_2)$ depend linearly on $d$, while $V_{[0,1]^d}^{HK}(f_5)$ depends exponentially on $d$.

\begin{itemize}

\item[(1)] A linear function with vanishing mixed derivates:
\begin{equation}\label{f1}
f_1(\bv) = \sum_{i=1}^{d} v_i, \quad V_{[0,1]^d}^{HK}(f_1) = \sum_{i=1}^d \int_{0}^1 \D v_i = d.
\end{equation}

\item[(2)] A quadratic function with vanishing mixed derivates:
\begin{equation}\label{f2}
f_2(\bv) = \sum_{i=1}^{d} v_{i}^2, \quad V_{[0,1]^d}^{HK}(f_2) = 2 \sum_{i=1}^d \int_{0}^1 v_i \D v_i = d.
\end{equation}

\item[(3)] A continuous function with discontinuous derivates:
\begin{equation}
f_3(\bv) = \sum_{i=1}^{d} |v_{i}|.
\end{equation}

\item[(4)] A continuous function with discontinuous derivates:
\begin{equation}
f_4(\bv) = \max_{1\le i\le d}  v_{i}.
\end{equation}

\item[(5)] A continuous function with its variation depending on $d$ exponentially:
\begin{equation}\label{f5}
f_5(\bv) = \prod_{i=1}^{d} v_{i}, \quad V_{[0,1]^d}^{HK}(f_5) = \sum_{l=1}^d \sum_{F_l} \left(\int_0^1 \D v_i\right)^{l} = 2^d -1. 
\end{equation}


\end{itemize}

\subsection{Implementation}

{\bf Sampling:} 
The point distributions of $(\bx_1, \dots, \bx_P)$ and $(\by_1, \dots, \by_M)$ are obtained by the Markov Chain Monte Carlo (MCMC),
where a simple Gaussian random walk Metropolis sampling is adopted to draw from $|g(\bv)|$ with standard deviation $0.1$ \cite{bk:RobertCasella2004}. Under different settings of dimensionality $d$ and rank $m$, we need to adjust the parameter $\epsilon$ to ensure a high acceptance ratio in the Metropolis sampling and avoid the over-repetitions of samples.

{\bf Discrepancy:} 
The calculation of the star discrepancy is a NP-hard problem \cite{GnewuchWahlstromWinzen2012}. Nonetheless, it can be realized by an efficient algorithm based on threshold accepting and probabilistic sampling \cite{GnewuchWahlstromWinzen2012}. Here we set the total number of iterations $I=128$ and the trial time $T = 5$ for each test in order to ensure a satisfactory approximation.

{\bf Matching:} 
Two kinds of matching strategies are tested. One is the random matching and the other  the optimal assignment. The random matching is realized by sorting a randomly generated sequence in an increasing order. The optimal assignment is realized by the Hungarian algorithm \cite{Kuhn1955} where the distance is chosen as the Manhattan distance. Although the complexity of the Hungarian algorithm scales polynomially in the number of points, its cost still grows dramatically when $P_k$ and $M_k$ become large. Thus for the sake of comparison, we only use the optimal assignment for the experiments of relatively small sample size ($N_{tot} \le 10^6)$.

{\bf Parallellization:} 
For each group of experiments, we use $64$ threads to generate samples from mutually independent Markov chains and collect all data through Message Passing Interference (MPI) standard. In order to realize the sequential clustering in a distributed platform, we decompose the domain into mutually disjoint 64 bins via balancing the particle numbers in each processor, which in turn strikes a balance in overload. All the numerical examples are obtained with our own Fortran implementations and run on the High-Performance Computing Platform of Peking University with the platform: 2$\times$Intel Xeon E5-2697A v4 (2.60GHz,
40MB Cache, 9.6GT/s QPI Speed, 16 Cores, 32 Threads) and 256GB Memory.

\subsection{Numerical results}

To clarify the efficiency of SPADE, we calculate the ratio $N_{tot}^A/N_{tot}$ and relative errors for all experiments with $\vartheta = 0.08$ in Table \ref{tab_relative_error}, where $N_{tot}$ and $N_{tot}^A$ denote the total particle number before and after annihilation, respectively. The partition levels $K$ and the total wall time  in sequential clustering, under different  $N_{tot}$, $N$, $\vartheta$ and $d$, are collected in Table \ref{tab_K} and Table \ref{tab_time}, respectively.

For the readers' convenience, all the raw data, ranging from $d=12$ to $d=1080$, are collected in Tables \ref{tab_data12d}-\ref{tab_data1080d} in the appendix as supplementary materials.
There $\textup{Rand(1)}$ corresponds to the results obtained by one random matching in each cluster, $\textup{Rand(100)}$ the averaged results of 100 random matchings, 
and $\textup{Hungarian}$ the results produced by the optimal assignment. 
The parameter $\epsilon$ and the acceptance ratio in sampling, the parameter $\vartheta$ in adaptive clustering, the positive particle number $P$, the negative particle number $M$ and the estimators $I_{-}(f)$ for all test functions are recorded. The results with relative error exceeding $5\%$ are marked in bold font.

{\bf Clustering:} We first investigate the performance of adaptive clustering. According to Table \ref{tab_K}, the partition level $K$ depends on the parameter $\vartheta$, the total sample size $N_{tot}$ and the dimensionality $d$. By increasing $\vartheta$, it will dramatically reduce the partition level $K$ so that more particles are annihilated. Too small $\vartheta$ is not suggested as it will lead to a very large $K$, resulting in an inefficient annihilation and a severe over-fitting problem (sometimes $K$ even exceeds the total sample size $N_{tot}$). In addition, we observe that for a fixed $\vartheta$, there is a positive correlation between partition level $K$ and total sample size $N_{tot}$. In particular, when $\vartheta = 0.08$, $K$ is almost linearly proportional to $\sqrt{N_{tot}}$.



{\bf Accuracy and variation:} It is readily observed that the accuracy of SPADE is strongly influenced by the variation of test functions. As presented in Table \ref{tab_relative_error}, for $f_1$, $f_2$, $f_3$ and $f_4$, regardless of the continuity of their derivates, SPADE can achieve a satisfactory accuracy even after removing more than $70\%$ of total particles (see the groups $d=36$, $N_{tot} = 10^7$, $\vartheta = 0.08$ and $d=60$, $N_{tot} = 10^7$, $\vartheta = 0.08$). And accurate results are still obtained after removing nearly half of particles in very high dimensional problem (see the group $d = 360, N_{tot} = 1\times10^7$, $\vartheta = 0.08$). Therefore, for an appropriate class of test functions, SPADE seems to be immune to the curse of dimensionality. By contrast, a severe fluctuation of errors is observed for $f_5$, because its total variation depends exponentially on $d$. This coincides with the predictions of the deterministic bound \eqref{def.HK_bounded_error_bound} in Theorem \ref{thm.particle_annhilation_error_discrete_bound} and the random bound \eqref{random_expection_bound} in Theorem \ref{thm.particle_annhilation_error_random_bound_no_variance}, as the performances of $\mathcal{E}(f)$ are distinct for different $f$ even under the same sequences before and after annihilation. Nevertheless, the bound provided by \eqref{def.HK_bounded_error_bound} and  \eqref{random_expection_bound} sometimes turns out be a gross over-estimate. With an appropriate choice of parameters, SPADE may still be able to preserve the accuracy for test function $f_5$  in high-dimensional cases after removing more than $10\%$ of total particles (see Table~\ref{tab_relative_error} for the groups: $d=360$, $N_{tot} = 10^5$, $\vartheta = 0.08$, and $d=1080$, $N_{tot} = 10^6$, $\vartheta = 0.08$).

{\bf Matching:} 
We compare the performance of random matching and optimal assignment. With a good partitioning, the accuracy of optimal assignment can be ensured by the deterministic bounds \eqref{def.HK_bounded_error_bound} in Theorem \ref{thm.particle_annhilation_error_discrete_bound}. In general, the optimal assignment is less accurate than the random matching, especially when $d$ becomes large, because the distance in Euclidian space and its geometric structure are heavily influenced by dimensionality. In contrast, the random matching, regardless of once and many times, can improve the accuracy. This validates the improvement in Eq.~\eqref{random_expection_bound} compared to Eq.~\eqref{def.HK_bounded_error_bound} and the concentration bound \eqref{random_error_Bernstein_bound} in Theorem \ref{thm.particle_annhilation_error_random_bound_no_variance}. In practice, it suggests to do random matching in each cluster for just once to save the computational time, without loss of accuracy. In fact, the result produced by random matching once is sometimes superior to the average of $100$ independent trials.

{\bf Computational time:} 
We record the total wall time in second for sequential clustering in Table \ref{tab_time}, which occupies more than $90\%$ of total computational time. Three observations are listed as follows.
(1) For fixed $\vartheta$ and $N_{tot}$, the total wall time almost increases as $d$ increase. The main reason is that more directions are searched in calculating the star discrepancies  \cite{GnewuchWahlstromWinzen2012}. (2) For fixed $\vartheta$ and $d$, the total wall time is almost linearly dependent on the sample size $N_{tot}$, as well as on the partition level $K$. This observation is rather important in estimating the computational cost of sequential clustering. (3) For fixed $N_{tot}$ and $d$, the total wall time in general increases as $\vartheta$ is chosen smaller, along with a deeper level of the decision tree. But there are some exceptions. One can see that in the group $d = 1080$, $N_{tot} = 1\times 10^7$, the computational cost is the highest when $\vartheta = 0.08$, even though the partition level $K$ is the smallest.

%

\begin{table}[!h]
  \centering
  \caption{\small The ratio of annihilation $N_{tot}^A/N_{tot}$ and relative errors $\textup{r.e.}(f)$ for random matching ($\vartheta=0.08$) are presented. Satisfactory results are obtained for test functions $f_1$, $f_2$, $f_3$ and $f_4$. By contrast, there are a huge fluctuations on the results of $f_5$ due to its large total variation.}
\label{tab_relative_error}
 \begin{lrbox}{\tablebox}
  \begin{tabular}{ccccccccccccccc}
\hline\hline
$d$ & $N_{tot}$ & $N$ & $N_{tot}^{A}/N_{tot}$  & $\textup{r.e.}(f_1)$ & $\textup{r.e.}(f_2)$ & $\textup{r.e.}(f_3)$ & $\textup{r.e.}(f_4)$ &$\textup{r.e.}(f_5)$\\
\hline
\multirow{4}{*}{12}
&$1\times10^4$	&$2446$		&65.04\%	&  	0.2310\%		&	2.4128\%		&	1.1648\%		&	0.9044\%		&	3.1539\%\\
&$1\times10^5$	&$24836$		&70.41\%	&	0.6840\%		&	1.6277\%		&	0.8013\%		&	0.9407\%		&	36.8416\%\\
&$1\times10^6$	&$243326$	&72.12\%	&	0.5299\%		&	1.6665\%		&	0.8485\%		&	0.9953\%		&	64.5606\%\\
&$1\times10^7$	&$2403016$	&75.68\%	&	0.5669\%		&	1.5188\%		&	0.7540\%		&	1.0674\%		&	54.6809\%\\
\hline
\multirow{4}{*}{36}
&$1\times10^4$	&$1124$		&60.64\%	&  	1.3238\%		&	3.6452\%		&	1.9537\%		&	0.3123\%		&	202.7732\%\\
&$1\times10^5$	&$12518$		&37.76\%	&	1.0242\%		&	3.9932\%		&	2.0902\%		&	1.2407\%		&	5.4367\%\\
&$1\times10^6$	&$121666$	&28.21\%	&	0.4692\%		&	3.3817\%		&	1.7842\%		&	1.4046\%		&	128.7225\%\\
&$1\times10^7$	&$1203984$	&27.17\%	&	0.0691\%		&	2.7788\%		&	1.4717\%		&	1.0867\%		&	567.4215\%\\
\hline
\multirow{4}{*}{60}
&$1\times10^4$	&$260$		&79.36\%	&  	0.2880\%		&	0.5271\%		&	2.0039\%		&	1.7056\%		&	26.8976\%\\
&$1\times10^5$	&$2760$		&51.58\%	&	0.4399\%		&	2.4687\%		&	1.0438\%		&	1.1901\%		&	100.6408\%\\
&$1\times10^6$	&$39256$		&33.40\%	&	0.9677\%		&	0.1069\%		&	0.3126\%		&	0.5714\%		&	61.8996\%\\
&$1\times10^7$	&$357952$	&27.12\%	&	0.6646\%		&	0.7043\%		&	0.1110\%		&	0.4783\%		&	181.1446\%\\
\hline
\multirow{4}{*}{120}
&$1\times10^4$	&$138$		&89.86\%	&  	1.3588\%		&	0.3399\%		&	0.2148\%		&	1.1358\%		&	0.0625\%\\
&$1\times10^5$	&$988$		&49.27\%	&	1.4231\%		&	1.0547\%		&	0.2036\%		&	1.9710\%		&	3.4353\%\\
&$1\times10^6$	&$7138$		&47.44\%	&	1.2049\%		&	2.8825\%		&	1.4231\%		&	0.1650\%		&	100.5381\%\\
&$1\times10^7$	&$68004$		&34.53\%	&	0.0959\%		&	2.2392\%		&	1.0184\%		&	0.7300\%		&	87.7581\%\\
\hline
\multirow{4}{*}{360}
&$1\times10^4$	&$304$		&92.26\%	&  	0.4972\%		&	0.0593\%		&	0.1184\%		&	0.7640\%		&	6.56E-09\\
&$1\times10^5$	&$2768$		&87.59\%	&	0.2050\%		&	0.0941\%		&	0.0040\%		&	0.3796\%		&	0.0479\%\\
&$1\times10^6$	&$18604$		&73.62\%	&	0.0388\%		&	0.0372\%		&	0.0875\%		&	0.2325\%		&	3.4452\%\\
&$1\times10^7$	&$232048$	&54.03\%	&	0.0057\%		&	0.2312\%		&	0.0970\%		&	0.1581\%		&	92.0226\%\\
\hline
\multirow{4}{*}{1080}
&$1\times10^4$	&$276$		&90.44\%	&  	0.3292\%		&	0.1862\%		&	0.1006\%		&	0.1302\%		&	0.0069\%\\
&$1\times10^5$	&$3202$		&86.25\%	&	0.2112\%		&	0.0592\%		&	0.0138\%		&	0.1377\%		&	0.3230\%\\
&$1\times10^6$	&$31972$		&88.96\%	&	0.2052\%		&	0.1592\%		&	0.0945\%		&	0.4773\%		&	0.0372\%\\
&$1\times10^7$	&$182104$		&91.75\%	&	0.0646\%		&	0.0105\%		&	0.0191\%		&	0.0227\%		&	0.0419\%\\
\hline\hline
 \end{tabular}
\end{lrbox}
\scalebox{0.8}{\usebox{\tablebox}}
\end{table}

\begin{table}[!h]
  \centering
  \caption{\small The partition level $K$ in sequential clustering under different total sample size $N_{tot}$, normalizing constant $N$, parameter $\vartheta$ and $d$. A positive correlation between   $K$ and $N_{tot}$, and a negative correlation between $K$ and $\vartheta$ are observed.}
\label{tab_K}
 \begin{lrbox}{\tablebox}
  \begin{tabular}{cccccccc}
\hline\hline
$N_{tot}$	&$\vartheta$  &	$d=12$	&	$d=36$		&	$d=60$		&	$d=120$		&	$d=360$		&	$d=1080$	\\
\hline
\multirow{3}{*}{$1\times 10^4$}
&	0.005	&	11073	&	12120	&	12127	&	12757	&	13076	&	12835	\\
&	0.02		&	3886		&	6160		&	12127	&	12757	&	13076	&	12835	\\
&	0.08		&	1011		&	1499		&	3245		&	6701		&	13076	&	2609		\\
\hline
$N$		&$-$  &	$2446$	&	$1124$		&	$260$		&	$138$		&	$304$		&	$276$	\\
\hline\hline
$N_{tot}$	&$\vartheta$  &	$d=12$	&	$d=36$		&	$d=60$		&	$d=120$		&	$d=360$		&	$d=1080$	\\
\hline
\multirow{3}{*}{$1\times 10^5$}
&	0.005	&	61320	&	117348	&	122846	&	134491	&	140958	&	145016	\\
&	0.02		&	13472	&	17249	&	43761	&	134367	&	140958	&	37940	\\
&	0.08		&	3361		&	3807		&	8991		&	20395	&	39578	&	8940		\\
\hline
$N$		&$-$  &	$24836$	&	$12518$		&	$2760$		&	$988$		&	$2768$		&	$3202$	\\
\hline\hline
$N_{tot}$	&$\vartheta$  &	$d=12$	&	$d=36$		&	$d=60$		&	$d=120$		&	$d=360$		&	$d=1080$	\\
\hline
\multirow{3}{*}{$1\times 10^6$}
&	0.005	&	180704	&	265893	&	467951	&	1377992	&	1410399	&	543961	\\
&	0.02		&	44135	&	50567	&	108769	&	293752	&	755026	&	124158	\\
&	0.08		&	11112	&	11391	&	20377	&	55575	&	153531	&	30554	\\
\hline
$N$		&$-$  &	$243326$	&	$121666$		&	$39256$		&	$7138$		&	$18604$		&	$31972$	\\
\hline\hline
$N_{tot}$	&$\vartheta$  &	$d=12$	&	$d=36$		&	$d=60$		&	$d=120$		&	$d=360$		&	$d=1080$	\\
\hline
\multirow{3}{*}{$1\times 10^7$}
&	0.005	&	569881	&	738886	&	1630915	&	14691077	&	7969705	&	3207698	\\
&	0.02		&	146090	&	159723	&	318590	&	3161834	&	1876241	&	718666	\\
&	0.08		&	37162	&	37454	&	64692	&	160285	&	363425	&	185654	\\
\hline
$N$		&$-$  &	$2403016$	&	$1203984$		&	$357952$		&	$68004$		&	$232048$		&	$182104$	\\
\hline\hline
 \end{tabular}
\end{lrbox}
\scalebox{0.8}{\usebox{\tablebox}}
\end{table}

\begin{table}[!h]
  \centering
  \caption{\small The total wall time in second for sequential clustering, which occupies more than $90\%$ of total computational time. The computational time is almost linearly dependent on $N_{tot}$, and usually increases for a larger $d$ or a smaller $\vartheta$.}
\label{tab_time}
 \begin{lrbox}{\tablebox}
  \begin{tabular}{cccccccc}
\hline\hline
$N_{tot}$	&$\vartheta$  &	$d=12$	&	$d=36$		&	$d=60$		&	$d=120$		&	$d=360$		&	$d=1080$	\\
\hline
\multirow{3}{*}{$1\times 10^4$}
&	0.005	&	2.374E-01		&	6.182E-01		&	8.029E-01		&	2.236E+00	&	4.069E+00	&	7.544E+00	\\
&	0.02		&	1.341E+02	&	1.274E+01	&	8.022E-01		&	2.044E+00	&	4.048E+00	&	7.347E+00	\\
&	0.08		&	2.199E+02	&	5.184E+02	&	4.117E+02	&	2.798E+03	&	4.267E+00	&	1.157E+04	\\
\hline
\multirow{3}{*}{$1\times 10^5$}
&	0.005	&	1.581E+03	&	4.639E+03	&	6.166E+01	&	1.172E+02	&	9.531E+02	&	8.574E+02	\\
&	0.02		&	1.916E+03	&	3.542E+03	&	3.429E+03	&	7.101E+04	&	1.099E+03	&	8.889E+04	\\
&	0.08		&	1.394E+03	&	3.089E+03	&	6.313E+03	&	1.657E+04	&	5.234E+03	&	7.225E+04	\\
\hline
\multirow{3}{*}{$1\times 10^6$}
&	0.005	&	1.767E+04	&	4.799E+04	&	8.559E+03	&	1.049E+04	&	5.127E+04	&	8.466E+05	\\
&	0.02		&	1.246E+04	&	3.615E+04	&	6.357E+04	&	1.032E+05	&	2.034E+05	&	5.816E+05	\\
&	0.08		&	1.145E+04	&	1.994E+04	&	3.861E+04	&	1.170E+05	&	3.144E+05	&	6.110E+05	\\
\hline
\multirow{3}{*}{$1\times 10^7$}
&	0.005	&	1.254E+05	&	3.573E+05	&	5.225E+05	&	9.696E+05	&	1.525E+06	&	3.784E+06	\\
&	0.02		&	1.060E+05	&	2.550E+05	&	5.228E+05	&	5.742E+05	&	2.351E+06	&	3.824E+06	\\
&	0.08		&	9.874E+04	&	1.725E+05	&	3.194E+05	&	7.521E+05	&	2.455E+06	&	6.414E+06	\\
\hline\hline
 \end{tabular}
\end{lrbox}
\scalebox{0.8}{\usebox{\tablebox}}
\end{table}


%

\section{Conclusion and discussion}
\label{sec:conclusion}

We proposed an algorithm, dubbed Sequential-clustering Particle Annihilation via Discrepancy Estimation (SPADE),  
for efficiently removing particles with opposite weights in an empirical signed measure 
within an acceptable accuracy. SPADE first seeks adaptive clustering of particles via controlling their number-theoretic discrepancies, then pairs the positive and negative particles via random matching and finally removes the paired ones. It does not require any a priori knowledge of nodal hyper-surface of underlying integrand as adopted in fixed-node approximation, and alleviates the restriction of mesh size in grid-based annihilation, thereby providing a potential approach to overcoming the numerical sign problem in general settings. We have proved that both the deterministic and the random error bounds of SPADE are affected by two factors. One factor measures the irregularity of the point distribution 
via bounding the star discrepancy in each cluster,  
and is proportional to $n^{-1}$ ($n$ is the particle number in one cluster).
That is, it has the same denominator as the bound of low-discrepancy sequences in quasi-Monte Carlo.
Moreover, a unified  numerator is set to be $\vartheta \sqrt{N}$ ($\vartheta \in (0, 1)$ is the error-controlling parameter and $N$ is the normalizing constant), thereby avoiding the embarrassing scaling of $(\log n)^d$ in quasi-Monte Carlo ($d$ is the dimensionality).
The other factor is the variation of the test function,
and implicitly depends on $d$.  
Therefore SPADE can be immune to the curse of dimensionality for suitable test functions,
which has been validated by numerical experiments up to $d=1080$.


A direct application of SPADE is to alleviate the exponential growth of both particle number and stochastic variances in simulating classical transport and quantum many-body dynamics, especially those where long-time behaviors of several non-oscillating averaged quantities, or physical observables, are of the most interest. In fact, we have recently employed SPADE to alleviate the numerical sign problem in stochastic Wigner simulations. 
Numerical results in 6-D phase space are very promising \cite{XiongShao2020} and the generalization to 12-D problems is still ongoing.



\section{Supplementary: Raw data in numerical experiments}

\begin{table}
  \centering
  \caption{\small Numerical results for $d=12$, $m=4$, $\epsilon = 0.6$. Acceptance ratio in MCMC is $85\%$. The results with relative error exceeding $5\%$ are marked in bold font. 
}\label{tab_data12d}
 \begin{lrbox}{\tablebox}
  \begin{tabular}{ccccccccccc}
\hline\hline
Method & $\vartheta$ & $K$ & $P$ & $M$ & $I_{-}(f_1)$ & $I_{-}(f_2)$ & $I_{-}(f_3)$ & $I_{-}(f_4)$ & $I_{-}(f_5)$\\
\hline
\multicolumn{10}{c}{$N_{tot} = P+M = 1\times 10^4$, $N = P-M = 2446$}\\
\hline
Sample	&-	&-	&	6223	&	3777	&	3.973E+00	&	1.616E+01	&	1.112E+01	&	2.301E+00	&	-7.828E-02\\
\hline
\multirow{3}{*}{Rand(1)}
&	0.005	&	11073	&	5675	&	3229	&	3.914E+00	&	1.623E+01	&	1.112E+01	&	2.316E+00	&	{\bf -1.176E-01}\\
&	0.02	&	3886	&	5447	&	3001	&	3.945E+00	&	1.637E+01	&	1.119E+01	&	2.315E+00	&	{\bf -1.102E-01}\\
&	0.08	&	1011	&	4475	&	2029	&	3.964E+00	&	1.655E+01	&	1.125E+01	&	2.322E+00	&	-8.074E-02\\
\hline
\multirow{3}{*}{Rand(100)}
&	0.005	&	11073	&	5675	&	3229	&	3.934E+00	&	1.622E+01	&	1.113E+01	&	2.309E+00	&	{\bf -1.167E-01}\\
&	0.02	&	3886	&	3001	&	2446	&	3.923E+00	&	1.634E+01	&	1.117E+01	&	2.314E+00	&	{\bf -1.031E-01}\\
&	0.08	&	1011	&	4475	&	2029	&	3.969E+00	&	1.653E+01	&	1.124E+01	&	2.323E+00	&	{\bf -6.464E-02}\\
\hline
\multirow{3}{*}{Hungarian}
&	0.005	&	11073	&	5675	&	3229	&	3.934E+00	&	1.625E+01	&	1.113E+01	&	2.313E+00	&	{\bf -7.326E-02}\\
&	0.02	&	3886	&	3001	&	2446	&	3.915E+00	&	1.637E+01	&	1.117E+01	&	2.320E+00	&	{\bf -8.958E-02}\\
&	0.08	&	1011	&	4475	&	2029	&	4.012E+00	&	1.662E+01	&	1.127E+01	&	2.332E+00	&	{\bf -4.762E-02}\\
\hline\hline
\multicolumn{10}{c}{$N_{tot} = P+M = 1\times 10^5$, $N = P-M = 24836$}\\
\hline
Sample	&-	&-	&	62418	&	37582	&	3.956E+00	&	1.608E+01	&	1.106E+01	&	2.279E+00	&	-1.585E-01\\
\hline
\multirow{3}{*}{Rand(1)}
&	0.005	&	61320	&	58740	&	33904	&	3.955E+00	&	1.611E+01	&	1.107E+01	&	2.282E+00	&	{\bf -1.416E-01}\\
&	0.02	&	13472	&	53684	&	28848	&	3.981E+00	&	1.616E+01	&	1.108E+01	&	2.290E+00	&	{\bf -1.283E-01}\\
&	0.08	&	3361	&	47624	&	22788	&	3.983E+00	&	1.634E+01	&	1.115E+01	&	2.301E+00	&	{\bf -1.001E-01}\\
\hline
\multirow{3}{*}{Rand(100)}
&	0.005	&	61320	&	58740	&	33904	&	3.951E+00	&	1.609E+01	&	1.106E+01	&	2.282E+00	&	-1.532E-01\\
&	0.02	&	23532	&	53684	&	28848	&	3.985E+00	&	1.617E+01	&	1.109E+01	&	2.291E+00	&	{\bf -1.392E-01}\\
&	0.08	&	3361	&	47624	&	22788	&	4.007E+00	&	1.638E+01	&	1.116E+01	&	2.303E+00	&	{\bf -8.162E-02}\\
\hline
\multirow{3}{*}{Hungarian}
&	0.005	&	61320	&	58740	&	33904	&	3.962E+00	&	1.610E+01	&	1.107E+01	&	2.283E+00	&	{\bf -1.479E-01}\\
&	0.02	&	23532	&	53684	&	28848	&	3.989E+00	&	1.622E+01	&	1.110E+01	&	2.296E+00	&	-1.524E-01\\
&	0.08	&	3361	&	47624	&	22788	&	4.018E+00	&	1.637E+01	&	1.116E+01	&	2.304E+00	&	{\bf -1.064E-01}\\
\hline\hline
\multicolumn{10}{c}{$N_{tot} = P+M = 1\times 10^6$, $N = P-M = 243326$}\\
\hline
Sample	&-	&-	&	621663	&	378337	&	3.970E+00	&	1.595E+01	&	1.102E+01	&	2.271E+00	&	-1.352E-02\\
\hline
\multirow{3}{*}{Rand(1)}
&	0.005	&	180701	&	571698	&	328372	&	3.968E+00	&	1.598E+01	&	1.103E+01	&	2.276E+00	&	{\bf -1.164E-02}\\
&	0.02	&	44135	&	530335	&	287009	&	3.974E+00	&	1.606E+01	&	1.105E+01	&	2.283E+00	&	-1.344E-02\\
&	0.08	&	11112	&	482285	&	238959	&	3.949E+00	&	1.622E+01	&	1.112E+01	&	2.294E+00	&	{\bf -4.791E-03}\\
\hline
\multirow{3}{*}{Rand(100)}
&	0.005	&	180704	&	571698	&	328372	&	3.966E+00	&	1.598E+01	&	1.103E+01	&	2.276E+00	&	-1.300E-02\\
&	0.02	&	44135	&	530335	&	287009	&	3.970E+00	&	1.605E+01	&	1.105E+01	&	2.283E+00	&	-1.385E-02\\
&	0.08	&	11112	&	482285	&	238959	&	3.952E+00	&	1.622E+01	&	1.112E+01	&	2.294E+00	&	{\bf 1.502E-03}\\
\hline
\multirow{3}{*}{Hungarian}
&	0.005	&	180704	&	571698	&	328372	&	3.966E+00	&	1.598E+01	&	1.103E+01	&	2.277E+00	&	-1.331E-02\\
&	0.02	&	44135	&	530335	&	287009	&	3.972E+00	&	1.606E+01	&	1.106E+01	&	2.284E+00	&	{\bf -2.381E-02}\\
&	0.08	&	11112	&	482285	&	238959	&	3.953E+00	&	1.621E+01	&	1.111E+01	&	2.294E+00	&	{\bf -4.567E-03}\\
\hline\hline
\multicolumn{10}{c}{$N_{tot} = P+M = 1\times 10^7$, $N = P-M = 2403016$}\\
\hline
Sample	&-	&-	&	6201508	&	3798492	&	3.969E+00	&	1.588E+01	&	1.101E+01	&	2.263E+00	&	1.373E-02\\
\hline
\multirow{3}{*}{Rand(1)}
&	0.005	&	569881	&	5683933	&	3280917	&	3.962E+00	&	1.591E+01	&	1.101E+01	&	2.269E+00	&	{\bf 1.037E-02}\\
&	0.02	&	146090	&	5380143	&	2977127	&	3.959E+00	&	1.597E+01	&	1.103E+01	&	2.274E+00	&	{\bf 1.161E-02}\\
&	0.08	&	37162	&	4985561	&	2582545	&	3.947E+00	&	1.612E+01	&	1.109E+01	&	2.287E+00	&	{\bf 6.224E-03}\\
\hline
\multirow{3}{*}{Rand(100)}
&	0.005	&	569881	&	5683933	&	3280917	&	3.962E+00	&	1.592E+01	&	1.102E+01	&	2.269E+00	&	{\bf 1.278E-02}\\
&	0.02	&	146090	&	5380143	&	2977127	&	3.958E+00	&	1.598E+01	&	1.103E+01	&	2.275E+00	&	{\bf 9.972E-03}\\
&	0.08	&	37162	&	4985561	&	2582545	&	3.947E+00	&	1.613E+01	&	1.109E+01	&	2.288E+00	&	{\bf 6.970E-03}\\
\hline\hline
 \end{tabular}
\end{lrbox}
\scalebox{0.68}{\usebox{\tablebox}}
\end{table}

\begin{table}
  \centering
  \caption{\small Numerical results for $d=36$, $m=12$, $\epsilon = 0.3$. Acceptance ratio in MCMC is $75\%$. The results with relative error exceeding $5\%$ are marked in bold font.
}\label{tab_data36d}
 \begin{lrbox}{\tablebox}
  \begin{tabular}{ccccccccccc}
\hline\hline
Method & $\vartheta$ & $K$ & $P$ & $M$ & $I_{-}(f_1)$ & $I_{-}(f_2)$ & $I_{-}(f_3)$ & $I_{-}(f_4)$ & $I_{-}(f_5)$\\
\hline
\multicolumn{10}{c}{$N_{tot} = P+M = 1\times 10^4$, $N = P-M = 1124$}\\
\hline
Sample	&-	&-	&	5562	&	4438	&	2.300E+01	&	6.681E+01	&	3.901E+01	&	3.314E+00	&	3.793E+01\\
\hline
\multirow{3}{*}{Rand(1)}
&	0.005	&	12120	&	4944	&	3820	&	2.335E+01	&	6.802E+01	&	3.943E+01	&	3.341E+00	&	{\bf -3.203E+01}\\
&	0.02	&	6160	&	4820	&	3696	&	2.333E+01	&	6.760E+01	&	3.930E+01	&	3.333E+00	&	3.764E+01\\
&	0.08	&	1499	&	3594	&	2470	&	2.330E+01	&	6.925E+01	&	3.977E+01	&	3.324E+00	&	{\bf -3.898E+01}\\
\hline
\multirow{3}{*}{Rand(100)}
&	0.005	&	12120	&	4944	&	3820	&	2.327E+01	&	6.758E+01	&	3.928E+01	&	3.331E+00	&	{\bf 2.524E+01}\\
&	0.02	&	6160	&	4820	&	3696	&	2.335E+01	&	6.776E+01	&	3.934E+01	&	3.333E+00	&	{\bf 2.204E+01}\\
&	0.08	&	1499	&	3594	&	2470	&	2.337E+01	&	6.902E+01	&	3.973E+01	&	3.325E+00	&	{\bf -1.649E+01}\\
\hline
\multirow{3}{*}{Hungarian}
&	0.005	&	12120	&	4944	&	3820	&	2.342E+01	&	6.846E+01	&	3.954E+01	&	3.349E+00	&	3.787E+01\\
&	0.02	&	6160	&	4820	&	3696	&	2.349E+01	&	6.862E+01	&	3.960E+01	&	3.350E+00	&	3.796E+01\\
&	0.08	&	1499	&	3594	&	2470	&	2.349E+01	&	6.986E+01	&	3.997E+01	&	3.345E+00	&	{\bf -3.794E+01}\\
\hline\hline
\multicolumn{10}{c}{$N_{tot} = P+M = 1\times 10^5$, $N = P-M = 12518$}\\
\hline
Sample	&-	&-	&	56259	&	43741	&	2.361E+01	&	6.768E+01	&	3.934E+01	&	3.293E+00	&	-8.300E+01\\
\hline
\multirow{3}{*}{Rand(1)}
&	0.005	&	117348	&	50027	&	37509	&	2.368E+01	&	6.820E+01	&	3.951E+01	&	3.303E+00	&	-8.289E+01\\
&	0.02	&	17249	&	38458	&	25940	&	2.373E+01	&	6.930E+01	&	3.986E+01	&	3.315E+00	&	-8.098E+01\\
&	0.08	&	3807	&	25140	&	12622	&	2.385E+01	&	7.038E+01	&	4.017E+01	&	3.334E+00	&	{\bf -8.751E+01}\\
\hline
\multirow{3}{*}{Rand(100)}
&	0.005	&	117348	&	50027	&	37509	&	2.364E+01	&	6.817E+01	&	3.951E+01	&	3.298E+00	&	-8.353E+01\\
&	0.02	&	17249	&	38458	&	25940	&	2.371E+01	&	6.919E+01	&	3.983E+01	&	3.308E+00	&	{\bf -5.103E+01}\\
&	0.08	&	3807	&	25140	&	12622	&	2.390E+01	&	7.032E+01	&	4.016E+01	&	3.331E+00	&	{\bf -3.138E+01}\\
\hline
\multirow{3}{*}{Hungarian}
&	0.005	&	117348	&	50027	&	37509	&	2.363E+01	&	6.840E+01	&	3.958E+01	&	3.303E+00	&	-8.201E+01\\
&	0.02	&	17249	&	38458	&	25940	&	2.372E+01	&	7.194E+01	&	4.067E+01	&	3.362E+00	&	-8.115E+01\\
&	0.08	&	3807	&	25140	&	12622	&	2.395E+01	&	7.179E+01	&	4.062E+01	&	3.360E+00	&	{\bf -7.827E+01}\\
\hline\hline
\multicolumn{10}{c}{$N_{tot} = P+M = 1\times 10^6$, $N = P-M = 121666$}\\
\hline
Sample	&-	&-	&	560833	&	439167	&	2.376E+01	&	6.783E+01	&	3.941E+01	&	3.281E+00	&	2.842E+01\\
\hline
\multirow{3}{*}{Rand(1)}
&	0.005	&	265893	&	424226	&	302560	&	2.381E+01	&	6.870E+01	&	3.968E+01	&	3.300E+00	&	{\bf -8.343E-02}\\
&	0.02	&	50567	&	276064	&	154398	&	2.385E+01	&	6.976E+01	&	4.000E+01	&	3.321E+00	&	{\bf -8.660E+00}\\
&	0.08	&	11391	&	201906	&	80240	&	2.387E+01	&	7.012E+01	&	4.011E+01	&	3.328E+00	&	{\bf -8.163E+00}\\
\hline
\multirow{3}{*}{Rand(100)}
&	0.005	&	265893	&	424226	&	302560	&	2.378E+01	&	6.868E+01	&	3.968E+01	&	3.300E+00	&	{\bf 8.456E+00}\\
&	0.02	&	50567	&	276064	&	154398	&	2.382E+01	&	6.967E+01	&	3.998E+01	&	3.319E+00	&	{\bf 2.572E+00}\\
&	0.08	&	11391	&	201906	&	80240	&	2.387E+01	&	7.017E+01	&	4.012E+01	&	3.329E+00	&	{\bf -5.450E+00}\\
\hline
\multirow{3}{*}{Hungarian}
&	0.005	&	265893	&	424226	&	302560	&	2.378E+01	&	6.886E+01	&	3.973E+01	&	3.303E+00	&	{\bf 6.665E+00}\\
&	0.02	&	50567	&	276064	&	154398	&	2.380E+01	&	7.058E+01	&	4.024E+01	&	3.343E+00	&	{\bf 8.353E+00}\\
&	0.08	&	11391	&	201906	&	80240	&	2.383E+01	&	7.231E+01	&	4.075E+01	&	3.383E+00	&	{\bf -1.907E+01}\\
\hline\hline
\multicolumn{10}{c}{$N_{tot} = P+M = 1\times 10^7$, $N = P-M = 1203984$}\\
\hline
Sample	&-	&-	&	5601992	&	4398008	&	2.387E+01	&	6.788E+01	&	3.944E+01	&	3.285E+00	&	1.076E+00\\
\hline
\multirow{3}{*}{Rand(1)}
&	0.005	&	738886	&	3302812	&	2098828	&	2.390E+01	&	6.905E+01	&	3.981E+01	&	3.307E+00	&	{\bf 5.264E+00}\\
&	0.02	&	159723	&	2419796	&	1215812	&	2.391E+01	&	6.952E+01	&	3.995E+01	&	3.316E+00	&	{\bf 8.713E+00}\\
&	0.08	&	37454	&	1960715	&	756731	&	2.388E+01	&	6.977E+01	&	4.002E+01	&	3.320E+00	&	{\bf -5.029E+00}\\
\hline
\multirow{3}{*}{Rand(100)}
&	0.005	&	738886	&	3302812	&	2098828	&	2.390E+01	&	6.906E+01	&	3.981E+01	&	3.307E+00	&	{\bf 1.032E+01}\\
&	0.02	&	159723	&	2419796	&	1215812	&	2.390E+01	&	6.949E+01	&	3.994E+01	&	3.314E+00	&	{\bf 6.796E+00}\\
&	0.08	&	37454	&	1960715	&	756731	&	2.388E+01	&	6.977E+01	&	4.003E+01	&	3.320E+00	&	{\bf -2.767E+00}\\
\hline\hline
 \end{tabular}
\end{lrbox}
\scalebox{0.68}{\usebox{\tablebox}}
\end{table}

\begin{table}
  \centering
  \caption{\small Numerical results for $d=60$, $m=20$, $\epsilon = 0.2$. Acceptance ratio in MCMC is $76\%$. The results with relative error exceeding $5\%$ are marked in bold font.
}\label{tab_data60d}
 \begin{lrbox}{\tablebox}
  \begin{tabular}{ccccccccccc}
\hline\hline
Method & $\vartheta$ & $K$ & $P$ & $M$ & $I_{-}(f_1)$ & $I_{-}(f_2)$ & $I_{-}(f_3)$ & $I_{-}(f_4)$ & $I_{-}(f_5)$\\
\hline
\multicolumn{10}{c}{$N_{tot} = P+M = 1\times 10^4$, $N = P-M = 260$}\\
\hline
Sample	&-	&-	&	5130	&	4870	&	5.858E+01	&	1.900E+02	&	8.213E+01	&	5.174E+00	&	6.616E+04\\
\hline
\multirow{3}{*}{Rand(1)}
&	0.005	&	12127	&	4572	&	4312	&	5.793E+01	&	1.851E+02	&	8.054E+01	&	5.188E+00	&	{\bf 2.364E+04}\\
&	0.02	&	12127	&	4572	&	4312	&	5.873E+01	&	1.859E+02	&	8.059E+01	&	5.174E+00	&	6.714E+04\\
&	0.08	&	3245	&	4098	&	3838	&	5.875E+01	&	1.890E+02	&	8.048E+01	&	5.262E+00	&	{\bf 8.396E+04}\\
\hline
\multirow{3}{*}{Rand(100)}
&	0.005	&	12127	&	4572	&	4312	&	5.783E+01	&	1.846E+02	&	8.046E+01	&	5.122E+00	&	6.328E+04\\
&	0.02	&	12127	&	4572	&	4312	&	5.793E+01	&	1.843E+02	&	8.038E+01	&	5.115E+00	&	6.691E+04\\
&	0.08	&	3245	&	4098	&	3838	&	5.880E+01	&	1.880E+02	&	8.057E+01	&	5.133E+00	&	{\bf 4.321E+04}\\
\hline
\multirow{3}{*}{Hungarian}
&	0.005	&	12127	&	4572	&	4312	&	5.786E+01	&	1.874E+02	&	8.095E+01	&	5.202E+00	&	6.496E+04\\
&	0.02	&	12127	&	4572	&	4312	&	5.786E+01	&	1.874E+02	&	8.095E+01	&	5.202E+00	&	6.496E+04\\
&	0.08	&	3245	&	4098	&	3838	&	5.880E+01	&	1.912E+02	&	8.119E+01	&	5.204E+00	&	6.776E+04\\
\hline\hline
\multicolumn{10}{c}{$N_{tot} = P+M = 1\times 10^5$, $N = P-M = 2760$}\\
\hline
Sample	&-	&-	&	51380	&	48620	&	6.119E+01	&	1.996E+02	&	8.384E+01	&	5.162E+00	&	-3.034E+05\\
\hline
\multirow{3}{*}{Rand(1)}
&	0.005	&	122846	&	45894	&	43134	&	6.151E+01	&	2.022E+02	&	8.427E+01	&	5.203E+00	&	{\bf -3.310E+05}\\
&	0.02	&	43761	&	43277	&	40517	&	5.998E+01	&	2.006E+02	&	8.392E+01	&	5.174E+00	&	-3.015E+05\\
&	0.08	&	8991	&	27170	&	24410	&	6.145E+01	&	2.045E+02	&	8.472E+01	&	5.224E+00	&	{\bf 1.944E+03}\\
\hline
\multirow{3}{*}{Rand(100)}
&	0.005	&	122846	&	45894	&	43134	&	6.133E+01	&	2.012E+02	&	8.413E+01	&	5.184E+00	&	{\bf -3.233E+05}\\
&	0.02	&	43761	&	43277	&	40517	&	6.147E+01	&	2.011E+02	&	8.410E+01	&	5.174E+00	&	-3.073E+05\\
&	0.08	&	8991	&	27170	&	24410	&	6.146E+01	&	2.042E+02	&	8.465E+01	&	5.215E+00	&	{\bf 5.232E+03}\\
\hline
\multirow{3}{*}{Hungarian}
&	0.005	&	122846	&	45894	&	43134	&	6.109E+01	&	2.015E+02	&	8.422E+01	&	5.194E+00	&	-3.170E+05\\
&	0.02	&	43761	&	43277	&	40517	&	6.129E+01	&	2.017E+02	&	8.422E+01	&	5.187E+00	&	-3.135E+05\\
&	0.08	&	8991	&	27170	&	24410	&	6.172E+01	&	2.073E+02	&	8.518E+01	&	5.269E+00	&	{\bf -1.630E+04}\\
\hline\hline
\multicolumn{10}{c}{$N_{tot} = P+M = 1\times 10^6$, $N = P-M = 39256$}\\
\hline
Sample	&-	&-	&	519628	&	480372	&	6.126E+01	&	2.039E+02	&	8.430E+01	&	5.212E+00	&	3.763E+05\\
\hline
\multirow{3}{*}{Rand(1)}
&	0.005	&	467951	&	448774	&	409518	&	6.116E+01	&	2.039E+02	&	8.424E+01	&	5.213E+00	&	{\bf 1.703E+05}\\
&	0.02	&	108769	&	315699	&	276443	&	6.087E+01	&	2.040E+02	&	8.420E+01	&	5.201E+00	&	{\bf 2.779E+05}\\
&	0.08	&	20377	&	186620	&	147364	&	6.067E+01	&	2.037E+02	&	8.404E+01	&	5.182E+00	&	{\bf 1.434E+05}\\
\hline
\multirow{3}{*}{Rand(100)}
&	0.005	&	467951	&	448774	&	409518	&	6.116E+01	&	2.040E+02	&	8.428E+01	&	5.208E+00	&	{\bf 1.728E+05}\\
&	0.02	&	108769	&	315699	&	276443	&	6.095E+01	&	2.042E+02	&	8.426E+01	&	5.205E+00	&	{\bf 1.806E+05}\\
&	0.08	&	20377	&	186620	&	147364	&	6.067E+01	&	2.039E+02	&	8.409E+01	&	5.186E+00	&	{\bf 1.512E+05}\\
\hline
\multirow{3}{*}{Hungarian}
&	0.005	&	467951	&	448774	&	409518	&	6.112E+01	&	2.045E+02	&	8.437E+01	&	5.216E+00	&	{\bf 1.688E+05}\\
&	0.02	&	108769	&	315699	&	276443	&	6.086E+01	&	2.047E+02	&	8.435E+01	&	5.207E+00	&	{\bf 1.428E+05}\\
&	0.08	&	20377	&	186620	&	147364	&	6.061E+01	&	2.069E+02	&	8.471E+01	&	5.229E+00	&	{\bf 9.173E+04}\\
\hline\hline
\multicolumn{10}{c}{$N_{tot} = P+M = 1\times 10^7$, $N = P-M = 357952$}\\
\hline
Sample	&-	&-	&	5178976	&	4821024	&	6.133E+01	&	2.040E+02	&	8.419E+01	&	5.146E+00	&	5.446E+04\\
\hline
\multirow{3}{*}{Rand(1)}
&	0.005	&	1630915	&	3862236	&	3504284	&	6.126E+01	&	2.050E+02	&	8.435E+01	&	5.154E+00	&	{\bf 3.798E+03}\\
&	0.02	&	318590	&	2479225	&	2121273	&	6.109E+01	&	2.057E+02	&	8.441E+01	&	5.170E+00	&	{\bf 1.482E+04}\\
&	0.08	&	64692	&	1534956	&	1177004	&	6.092E+01	&	2.054E+02	&	8.428E+01	&	5.171E+00	&	{\bf -4.419E+04}\\
\hline
\multirow{3}{*}{Rand(100)}
&	0.005	&	1630915	&	3862236	&	3504284	&	6.123E+01	&	2.049E+02	&	8.435E+01	&	5.155E+00	&	{\bf 2.305E+04}\\
&	0.02	&	318590	&	2479225	&	2121273	&	6.110E+01	&	2.057E+02	&	8.442E+01	&	5.172E+00	&	{\bf 7.535E+02}\\
&	0.08	&	64692	&	1534956	&	1177004	&	6.094E+01	&	2.054E+02	&	8.428E+01	&	5.172E+00	&	{\bf -2.793E+04}\\
\hline\hline
 \end{tabular}
\end{lrbox}
\scalebox{0.68}{\usebox{\tablebox}}
\end{table}

\begin{table}
  \centering
  \caption{\small Numerical results for $d=120$, $m=40$, $\epsilon = 0.05$. Acceptance ratio in MCMC is $77\%$. The results with relative error exceeding $5\%$ are marked in bold font.
}\label{tab_data120d}
 \begin{lrbox}{\tablebox}
  \begin{tabular}{ccccccccccc}
\hline\hline
Method & $\vartheta$ & $K$ & $P$ & $M$ & $I_{-}(f_1)$ & $I_{-}(f_2)$ & $I_{-}(f_3)$ & $I_{-}(f_4)$ & $I_{-}(f_5)$\\
\hline
\multicolumn{10}{c}{$N_{tot} = P+M = 1\times 10^4$, $N = P-M = 138$}\\
\hline
Sample	&-	&-	&	5069	&	4931	&	5.523E+01	&	5.269E+02	&	1.874E+02	&	6.822E+00	&	-1.397E+10\\
\hline
\multirow{3}{*}{Rand(1)}
&	0.005	&	12757	&	4602	&	4464	&	5.581E+01	&	5.247E+02	&	1.869E+02	&	6.786E+00	&	-1.397E+10\\
&	0.02	&	12757	&	4602	&	4464	&	5.586E+01	&	5.228E+02	&	1.868E+02	&	6.736E+00	&	-1.395E+10\\
&	0.08	&	6701	&	4562	&	4424	&	5.599E+01	&	5.287E+02	&	1.878E+02	&	6.745E+00	&	-1.398E+10\\
\hline
\multirow{3}{*}{Rand(100)}
&	0.005	&	12757	&	4602	&	4464	&	5.592E+01	&	5.216E+02	&	1.863E+02	&	6.825E+00	&	-1.398E+10\\
&	0.02	&	12757	&	4602	&	4464	&	5.606E+01	&	5.223E+02	&	1.864E+02	&	6.822E+00	&	-1.398E+10\\
&	0.08	&	6701	&	4562	&	4424	&	5.501E+01	&	5.207E+02	&	1.858E+02	&	6.794E+00	&	-1.398E+10\\
\hline
\multirow{3}{*}{Hungarian}
&	0.005	&	12757	&	4602	&	4464	&	5.523E+01	&	5.272E+02	&	1.878E+02	&	6.924E+00	&	-1.395E+10\\
&	0.02	&	12757	&	4602	&	4464	&	5.523E+01	&	5.272E+02	&	1.878E+02	&	6.924E+00	&	-1.395E+10\\
&	0.08	&	6701	&	4562	&	4424	&	5.457E+01	&	5.260E+02	&	1.875E+02	&	6.905E+00	&	-1.395E+10\\
\hline\hline
\multicolumn{10}{c}{$N_{tot} = P+M = 1\times 10^5$, $N = P-M = 988$}\\
\hline
Sample	&-	&-	&	50494	&	49506	&	5.883E+01	&	5.123E+02	&	1.834E+02	&	6.651E+00	&	-2.227E+12\\
\hline
\multirow{3}{*}{Rand(1)}
&	0.005	&	134491	&	45961	&	44973	&	5.932E+01	&	5.113E+02	&	1.833E+02	&	6.614E+00	&	-2.213E+12\\
&	0.02	&	134367	&	45961	&	44973	&	5.912E+01	&	5.112E+02	&	1.836E+02	&	6.615E+00	&	-2.212E+12\\
&	0.08	&	20395	&	36460	&	35472	&	5.799E+01	&	5.069E+02	&	1.831E+02	&	6.520E+00	&	-2.303E+12\\
\hline
\multirow{3}{*}{Rand(100)}
&	0.005	&	134491	&	45961	&	44973	&	5.979E+01	&	5.126E+02	&	1.837E+02	&	6.606E+00	&	{\bf -1.964E+12}\\
&	0.02	&	134367	&	28469	&	27481	&	5.980E+01	&	5.124E+02	&	1.837E+02	&	6.613E+00	&	-2.148E+12\\
&	0.08	&	20395	&	36460	&	35472	&	5.908E+01	&	5.082E+02	&	1.834E+02	&	6.509E+00	&	-2.267E+12\\
\hline
\multirow{3}{*}{Hungarian}
&	0.005	&	134491	&	45961	&	44973	&	6.017E+01	&	5.132E+02	&	1.836E+02	&	6.670E+00	&	-2.212E+12\\
&	0.02	&	134367	&	28469	&	27481	&	6.017E+01	&	5.132E+02	&	1.836E+02	&	6.670E+00	&	-2.212E+12\\
&	0.08	&	20395	&	36460	&	35472	&	6.000E+01	&	5.049E+02	&	1.825E+02	&	6.566E+00	&	-2.302E+12\\
\hline\hline
\multicolumn{10}{c}{$N_{tot} = P+M = 1\times 10^6$, $N = P-M = 7138$}\\
\hline
Sample	&-	&-	&	503569	&	496431	&	6.020E+01	&	4.883E+02	&	1.808E+02	&	6.334E+00	&	-1.471E+14\\
\hline
\multirow{3}{*}{Rand(1)}
&	0.005	&	1377992	&	460754	&	453616	&	5.970E+01	&	4.903E+02	&	1.810E+02	&	6.350E+00	&	-1.517E+14\\
&	0.02	&	293752	&	405990	&	398852	&	6.049E+01	&	4.999E+02	&	1.829E+02	&	6.371E+00	&	-1.516E+14\\
&	0.08	&	55575	&	240772	&	233634	&	5.947E+01	&	5.024E+02	&	1.834E+02	&	6.324E+00	&	{\bf 7.914E+11}\\
\hline
\multirow{3}{*}{Rand(100)}
&	0.005	&	1377992	&	460754	&	453616	&	5.934E+01	&	4.918E+02	&	1.813E+02	&	6.371E+00	&	-1.517E+14\\
&	0.02	&	293752	&	405990	&	398852	&	6.015E+01	&	4.984E+02	&	1.826E+02	&	6.347E+00	&	-1.514E+14\\
&	0.08	&	55575	&	240772	&	233634	&	5.991E+01	&	5.028E+02	&	1.837E+02	&	6.293E+00	&	{\bf 1.385E+12}\\
\hline
\multirow{3}{*}{Hungarian}
&	0.005	&	1377992	&	460754	&	453616	&	5.885E+01	&	4.934E+02	&	1.815E+02	&	6.380E+00	&	-1.515E+14\\
&	0.02	&	293752	&	405990	&	398852	&	5.954E+01	&	4.998E+02	&	1.828E+02	&	6.365E+00	&	-1.511E+14\\
&	0.08	&	55575	&	240772	&	233634	&	5.931E+01	&	5.043E+02	&	1.834E+02	&	6.342E+00	&	{\bf 1.984E+12}\\
\hline\hline
\multicolumn{10}{c}{$N_{tot} = P+M = 1\times 10^7$, $N = P-M = 68004$}\\
\hline
Sample	&-	&-	&	5034002	&	4965998	&	5.873E+01	&	4.955E+02	&	1.832E+02	&	6.331E+00	&	1.563E+15\\
\hline
\multirow{3}{*}{Rand(1)}
&	0.005	&	14691077	&	4676731	&	4608727	&	5.893E+01	&	4.960E+02	&	1.834E+02	&	6.326E+00	&	1.560E+15\\
&	0.02	&	3161834	&	4258819	&	4190815	&	5.897E+01	&	4.975E+02	&	1.837E+02	&	6.325E+00	&	1.551E+15\\
&	0.08	&	160285	&	1760628	&	1692624	&	5.878E+01	&	5.066E+02	&	1.851E+02	&	6.285E+00	&	{\bf 1.914E+14}\\
\hline
\multirow{3}{*}{Rand(100)}
&	0.005	&	14691077	&	4676731	&	4608727	&	5.889E+01	&	4.962E+02	&	1.834E+02	&	6.329E+00	&	1.541E+15\\
&	0.02	&	3161834	&	4258819	&	4190815	&	5.888E+01	&	4.977E+02	&	1.837E+02	&	6.331E+00	&	1.524E+15\\
&	0.08	&	160285	&	1760628	&	1692624	&	5.880E+01	&	5.062E+02	&	1.851E+02	&	6.271E+00	&	{\bf 1.561E+14}\\
\hline\hline
 \end{tabular}
\end{lrbox}
\scalebox{0.68}{\usebox{\tablebox}}
\end{table}

\begin{table}
  \centering
  \caption{\small Numerical results for $d=360$, $m=120$, $\epsilon = 0.02$. Acceptance ratio in MCMC is $63\%$. The results with relative error exceeding $5\%$ are marked in bold font.
}\label{tab_data360d}
 \begin{lrbox}{\tablebox}
  \begin{tabular}{ccccccccccc}
\hline\hline
Method & $\vartheta$ & $K$ & $P$ & $M$ & $I_{-}(f_1)$ & $I_{-}(f_2)$ & $I_{-}(f_3)$ & $I_{-}(f_4)$ & $I_{-}(f_5)$\\
\hline
\multicolumn{10}{c}{$N_{tot} = P+M =  1\times 10^4$, $N = P-M = 304$}\\
\hline
Sample	&-	&-	&	5152	&	4848	&	3.417E+02	&	2.032E+03	&	6.813E+02	&	6.942E+00	&	1.056E+56\\
\hline
\multirow{3}{*}{Rand(1)}
&	0.005	&	13076	&	4765	&	4461	&	3.436E+02	&	2.031E+03	&	6.813E+02	&	6.880E+00	&	1.056E+56\\
&	0.02	&	13076	&	4765	&	4461	&	3.438E+02	&	2.031E+03	&	6.816E+02	&	6.911E+00	&	1.056E+56\\
&	0.08	&	13076	&	4765	&	4461	&	3.434E+02	&	2.034E+03	&	6.821E+02	&	6.889E+00	&	1.056E+56\\
\hline
\multirow{3}{*}{Rand(100)}
&	0.005	&	13076	&	4765	&	4461	&	3.433E+02	&	2.030E+03	&	6.813E+02	&	6.885E+00	&	{\bf 9.825E+55}\\
&	0.02	&	13076	&	4765	&	4461	&	3.434E+02	&	2.031E+03	&	6.815E+02	&	6.885E+00	&	{\bf 9.931E+55}\\
&	0.08	&	13076	&	4765	&	4461	&	3.433E+02	&	2.031E+03	&	6.815E+02	&	6.879E+00	&	{\bf 9.826E+55}\\
\hline
\multirow{3}{*}{Hungarian}
&	0.005	&	13076	&	4765	&	4461	&	3.433E+02	&	2.030E+03	&	6.811E+02	&	6.909E+00	&	1.056E+56\\
&	0.02	&	13076	&	4765	&	4461	&	3.433E+02	&	2.030E+03	&	6.811E+02	&	6.909E+00	&	1.056E+56\\
&	0.08	&	13076	&	4765	&	4461	&	3.433E+02	&	2.030E+03	&	6.811E+02	&	6.909E+00	&	1.056E+56\\
\hline\hline
\multicolumn{10}{c}{$N_{tot} = P+M = 1\times 10^5$, $N = P-M = 2768$}\\
\hline
Sample	&-	&-	&	51384	&	48616	&	3.393E+02	&	1.989E+03	&	6.734E+02	&	6.841E+00	&	-2.897E+56\\
\hline
\multirow{3}{*}{Rand(1)}
&	0.005	&	140958	&	47847	&	45079	&	3.391E+02	&	1.988E+03	&	6.733E+02	&	6.842E+00	&	-2.897E+56\\
&	0.02	&	140958	&	47847	&	45079	&	3.393E+02	&	1.990E+03	&	6.737E+02	&	6.835E+00	&	-2.897E+56\\
&	0.08	&	39578	&	45178	&	42410	&	3.400E+02	&	1.991E+03	&	6.734E+02	&	6.815E+00	&	-2.895E+56\\
\hline
\multirow{3}{*}{Rand(100)}
&	0.005	&	140958	&	47847	&	45079	&	3.389E+02	&	1.989E+03	&	6.734E+02	&	6.835E+00	&	-2.900E+56\\
&	0.02	&	140958	&	47847	&	45079	&	3.390E+02	&	1.990E+03	&	6.734E+02	&	6.833E+00	&	-2.898E+56\\
&	0.08	&	39578	&	45178	&	42410	&	3.394E+02	&	1.989E+03	&	6.731E+02	&	6.813E+00	&	-2.905E+56\\
\hline
\multirow{3}{*}{Hungarian}
&	0.005	&	140958	&	47847	&	45079	&	3.392E+02	&	1.989E+03	&	6.734E+02	&	6.830E+00	&	-2.897E+56\\
&	0.02	&	140958	&	47847	&	45079	&	3.392E+02	&	1.989E+03	&	6.734E+02	&	6.830E+00	&	-2.897E+56\\
&	0.08	&	39578	&	45178	&	42410	&	3.395E+02	&	1.989E+03	&	6.732E+02	&	6.810E+00	&	-2.897E+56\\
\hline\hline
\multicolumn{10}{c}{$N_{tot} = P+M = 1\times 10^6$, $N = P-M = 18604$}\\
\hline
Sample	&-	&-	&	509302	&	490698	&	3.415E+02	&	1.997E+03	&	6.762E+02	&	6.804E+00	&	3.128E+59\\
\hline
\multirow{3}{*}{Rand(1)}
&	0.005	&	1410399	&	476948	&	458344	&	3.417E+02	&	1.998E+03	&	6.764E+02	&	6.797E+00	&	3.128E+59\\
&	0.02	&	755026	&	474596	&	455992	&	3.415E+02	&	1.997E+03	&	6.762E+02	&	6.797E+00	&	3.128E+59\\
&	0.08	&	153531	&	377425	&	358821	&	3.417E+02	&	1.996E+03	&	6.756E+02	&	6.788E+00	&	3.020E+59\\
\hline
\multirow{3}{*}{Rand(100)}
&	0.005	&	1410399	&	476948	&	458344	&	3.417E+02	&	1.998E+03	&	6.764E+02	&	6.800E+00	&	3.115E+59\\
&	0.02	&	755026	&	474596	&	455992	&	3.417E+02	&	1.998E+03	&	6.763E+02	&	6.798E+00	&	3.117E+59\\
&	0.08	&	153531	&	377425	&	358821	&	3.415E+02	&	1.996E+03	&	6.756E+02	&	6.785E+00	&	3.004E+59\\
\hline
\multirow{3}{*}{Hungarian}
&	0.005	&	1410399	&	476948	&	458344	&	3.418E+02	&	1.998E+03	&	6.764E+02	&	6.802E+00	&	3.037E+59\\
&	0.02	&	755026	&	474596	&	455992	&	3.418E+02	&	1.998E+03	&	6.763E+02	&	6.800E+00	&	3.037E+59\\
&	0.08	&	153531	&	377425	&	358821	&	3.416E+02	&	1.996E+03	&	6.757E+02	&	6.791E+00	&	3.021E+59\\
\hline\hline
\multicolumn{10}{c}{$N_{tot} = P+M = 1\times 10^7$, $N = P-M = 232048$}\\
\hline
Sample	&-	&-	&	5116024	&	4883976	&	3.405E+02	&	1.987E+03	&	6.735E+02	&	6.768E+00	&	1.651E+60\\
\hline
\multirow{3}{*}{Rand(1)}
&	0.005	&	7969705	&	4851894	&	4619846	&	3.405E+02	&	1.988E+03	&	6.737E+02	&	6.767E+00	&	1.646E+60\\
&	0.02	&	1876241	&	4231671	&	3999623	&	3.405E+02	&	1.989E+03	&	6.739E+02	&	6.757E+00	&	1.637E+60\\
&	0.08	&	363425	&	2817699	&	2585651	&	3.405E+02	&	1.992E+03	&	6.742E+02	&	6.757E+00	&	{\bf 1.317E+59}\\
\hline
\multirow{3}{*}{Rand(100)}
&	0.005	&	7969705	&	4851894	&	4619846	&	3.405E+02	&	1.988E+03	&	6.737E+02	&	6.766E+00	&	1.639E+60\\
&	0.02	&	1876241	&	4231671	&	3999623	&	3.404E+02	&	1.989E+03	&	6.739E+02	&	6.758E+00	&	1.632E+60\\
&	0.08	&	363425	&	2817699	&	2585651	&	3.404E+02	&	1.992E+03	&	6.742E+02	&	6.757E+00	&	{\bf 1.946E+59}\\
\hline\hline
 \end{tabular}
\end{lrbox}
\scalebox{0.68}{\usebox{\tablebox}}
\end{table}

\begin{table}
  \centering
  \caption{\small Numerical results for $d=1080$, $m=120$, $\epsilon = 0.02$. Acceptance ratio in MCMC is $63\%$. The results with relative error exceeding $5\%$ are marked in bold font.
}\label{tab_data1080d}
 \begin{lrbox}{\tablebox}
  \begin{tabular}{ccccccccccc}
\hline\hline
Method & $\vartheta$ & $K$ & $P$ & $M$ & $I_{-}(f_1)$ & $I_{-}(f_2)$ & $I_{-}(f_3)$ & $I_{-}(f_4)$ & $I_{-}(f_5)$\\
\hline
\multicolumn{10}{c}{$N_{tot} = P+M = 1\times 10^4$, $N = P-M = 276$}\\
\hline
Sample	&-	&-	&	5262	&	4722	&	5.296E+02	&	6.366E+03	&	2.149E+03	&	7.026E+00	&	-3.207E+189\\
\hline
\multirow{3}{*}{Rand(1)}
&	0.005	&	12835	&	5026	&	4486	&	5.295E+02	&	6.369E+03	&	2.150E+03	&	7.029E+00	&	-3.207E+189\\
&	0.02	&	12835	&	5026	&	4486	&	5.292E+02	&	6.373E+03	&	2.151E+03	&	7.022E+00	&	-3.207E+189\\
&	0.08	&	2609	&	4792	&	4252	&	5.314E+02	&	6.377E+03	&	2.151E+03	&	7.035E+00	&	-3.206E+189\\
\hline
\multirow{3}{*}{Rand(100)}
&	0.005	&	12835	&	5026	&	4486	&	5.299E+02	&	6.370E+03	&	2.150E+03	&	7.028E+00	&	-3.207E+189\\
&	0.02	&	12835	&	5026	&	4486	&	5.299E+02	&	6.370E+03	&	2.150E+03	&	7.030E+00	&	-3.207E+189\\
&	0.08	&	2609	&	4792	&	4252	&	5.318E+02	&	6.373E+03	&	2.151E+03	&	7.053E+00	&	-3.206E+189\\
\hline
\multirow{3}{*}{Hungarian}
&	0.005	&	12835	&	5026	&	4486	&	5.292E+02	&	6.370E+03	&	2.150E+03	&	7.021E+00	&	-3.207E+189\\
&	0.02	&	12835	&	5026	&	4486	&	5.292E+02	&	6.370E+03	&	2.150E+03	&	7.021E+00	&	-3.207E+189\\
&	0.08	&	2609	&	4792	&	4252	&	5.311E+02	&	6.375E+03	&	2.151E+03	&	7.036E+00	&	-3.206E+189\\
\hline\hline
\multicolumn{10}{c}{$N_{tot} = P+M = 1\times 10^5$, $N = P-M = 3202$}\\
\hline
Sample	&-	&-	&	52782	&	47218	&	5.289E+02	&	6.364E+03	&	2.148E+03	&	6.996E+00	&	-6.283E+193\\
\hline
\multirow{3}{*}{Rand(1)}
&	0.005	&	145016	&	51905	&	46341	&	5.293E+02	&	6.363E+03	&	2.148E+03	&	7.001E+00	&	-6.283E+193\\
&	0.02	&	37940	&	51515	&	45951	&	5.296E+02	&	6.362E+03	&	2.148E+03	&	6.993E+00	&	-6.283E+193\\
&	0.08	&	8940	&	45909	&	40345	&	5.300E+02	&	6.360E+03	&	2.148E+03	&	7.005E+00	&	-6.262E+193\\
\hline
\multirow{3}{*}{Rand(100)}
&	0.005	&	145016	&	51905	&	46341	&	5.294E+02	&	6.364E+03	&	2.148E+03	&	6.996E+00	&	-6.283E+193\\
&	0.02	&	37940	&	51515	&	45951	&	5.295E+02	&	6.362E+03	&	2.148E+03	&	6.994E+00	&	-6.271E+193\\
&	0.08	&	8940	&	45909	&	40345	&	5.303E+02	&	6.360E+03	&	2.148E+03	&	7.004E+00	&	-6.266E+193\\
\hline
\multirow{3}{*}{Hungarian}
&	0.005	&	145016	&	51905	&	46341	&	5.292E+02	&	6.364E+03	&	2.148E+03	&	6.996E+00	&	-6.283E+193\\
&	0.02	&	37940	&	51515	&	45951	&	5.293E+02	&	6.362E+03	&	2.148E+03	&	6.995E+00	&	-6.283E+193\\
&	0.08	&	8940	&	45909	&	40345	&	5.302E+02	&	6.362E+03	&	2.148E+03	&	7.006E+00	&	-6.262E+193\\
\hline\hline
\multicolumn{10}{c}{$N_{tot} = P+M = 1\times 10^6$, $N = P-M = 31972$}\\
\hline
Sample	&-	&-	&	527822	&	472178	&	5.286E+02	&	6.364E+03	&	2.148E+03	&	6.997E+00	&	7.041E+195\\
\hline
\multirow{3}{*}{Rand(1)}
&	0.005	&	543961	&	525331	&	469687	&	5.287E+02	&	6.364E+03	&	2.148E+03	&	6.999E+00	&	7.041E+195\\
&	0.02	&	124158	&	517840	&	462196	&	5.288E+02	&	6.366E+03	&	2.149E+03	&	7.006E+00	&	7.041E+195\\
&	0.08	&	30554	&	472631	&	416987	&	5.297E+02	&	6.374E+03	&	2.150E+03	&	7.031E+00	&	7.044E+195\\
\hline
\multirow{3}{*}{Rand(100)}
&	0.005	&	543961	&	525331	&	469687	&	5.287E+02	&	6.364E+03	&	2.148E+03	&	6.999E+00	&	7.041E+195\\
&	0.02	&	124158	&	517840	&	462196	&	5.288E+02	&	6.366E+03	&	2.149E+03	&	7.006E+00	&	7.041E+195\\
&	0.08	&	30554	&	472631	&	416987	&	5.297E+02	&	6.374E+03	&	2.150E+03	&	7.031E+00	&	{\bf 8.807E+195}\\
\hline
\multirow{3}{*}{Hungarian}
&	0.005	&	543961	&	525331	&	469687	&	5.286E+02	&	6.364E+03	&	2.148E+03	&	6.999E+00	&	7.041E+195\\
&	0.02	&	124158	&	517840	&	462196	&	5.287E+02	&	6.366E+03	&	2.148E+03	&	7.006E+00	&	7.041E+195\\
&	0.08	&	30554	&	472631	&	416987	&	5.296E+02	&	6.372E+03	&	2.150E+03	&	7.033E+00	&	{\bf 9.248E+195}\\
\hline\hline
\multicolumn{10}{c}{$N_{tot} = P+M = 1\times 10^7$, $N = P-M = 182104$}\\
\hline
Sample	&-	&-	&	5858055	&	4141945	&	5.351E+02	&	6.427E+03	&	2.160E+03	&	7.157E+00	&	4.825E+198\\
\hline
\multirow{3}{*}{Rand(1)}
&	0.005	&	3207698	&	5843127	&	4127017	&	5.351E+02	&	6.427E+03	&	2.160E+03	&	7.157E+00	&	4.826E+198\\
&	0.02	&	718666	&	5790453	&	4074343	&	5.351E+02	&	6.427E+03	&	2.160E+03	&	7.157E+00	&	4.824E+198\\
&	0.08	&	185654	&	5445547	&	3729437	&	5.354E+02	&	6.427E+03	&	2.160E+03	&	7.158E+00	&	4.827E+198\\
\hline
\multirow{3}{*}{Rand(100)}
&	0.005	&	3207698	&	5843127	&	4127017	&	5.351E+02	&	6.427E+03	&	2.160E+03	&	7.157E+00	&	4.825E+198\\
&	0.02	&	718666	&	5790453	&	4074343	&	5.351E+02	&	6.427E+03	&	2.160E+03	&	7.157E+00	&	4.824E+198\\
&	0.08	&	185654	&	5445547	&	3729437	&	5.354E+02	&	6.427E+03	&	2.160E+03	&	7.158E+00	&	4.826E+198\\
\hline\hline
 \end{tabular}
\end{lrbox}
\scalebox{0.68}{\usebox{\tablebox}}
\end{table}

\end{document}